\documentclass{amsproc}
\usepackage{amssymb,amsmath,amsthm,enumerate,amscd,graphicx,color}
\usepackage{pdfsync}
 \usepackage[colorlinks=true, pdfstartview=FitV, linkcolor=blue, citecolor=blue, urlcolor=blue]{hyperref}
\numberwithin{equation}{section}
\newtheorem{thm}{Theorem}[subsection]
\newtheorem{lem}[thm]{Lemma}
\newtheorem{cor}[thm]{Corollary}
\newtheorem{prop}[thm]{Proposition}

\newcommand{\Lnum}[3]{\left[ \begin{matrix} #1 \ ; \  #2 \\ \ \ #3\   \end{matrix} \right]}
\newcommand{\g}{\widehat{\mathfrak g}}
\newcommand{\D}{\Delta}

\newcommand{\thmref}[1]{Theorem~\ref{#1}}
\newcommand{\propref}[1]{Proposition~\ref{#1}}

\newcommand{\secref}[1]{\S\ref{#1}}
\newcommand{\lemref}[1]{Lemma~\ref{#1}}
\newcommand{\eqnref}[1]{~(\ref{#1})}

\def\um{{\underline m}}
\def\ul{{\underline l}}

\begin{document}

%\centerline{\sc Imaginary Verma modules and Kashiwara algebras for $U_q(\bighat{\mathfrak{g}})$.}

%\date{}                                           % Activate to display a given date or no date

\title{}
\title{Kashiwara Algebras and Imaginary Verma  Modules  for $U_q(\widehat{\mathfrak{g}})$}
\author{ Ben Cox}
\author{Vyacheslav Futorny}
\author{Kailash C. Misra}
\address{Department of Mathematics \\
The Graduate School at the College of Charleston \\
66 George Street  \\
Charleston SC 29424, USA}\email{coxbl@cofc.edu}
\address{Department of Mathematics\\
 University of S\~ao Paulo\\
 S\~ao Paulo, Brazil}
 \email{futorny@ime.usp.br}
 \address{Department of Mathematics\\
 North Carolina State University\\
 Raleigh, NC 27695-8205, USA}
 \begin{abstract} We consider imaginary Verma modules for quantum affine algebra
 $U_q(\widehat{\mathfrak{g}})$, where $\widehat{\mathfrak{g}}$ is of type 1 i.e. of non-twisted type,  and construct Kashiwara type operators and the Kashiwara algebra $\mathcal K_q$. We show that  a certain quotient  $\mathcal N_q^-$ of
 $U_q(\widehat{\mathfrak{g}})$ is a simple $\mathcal K_q$-module..
\end{abstract}
\date{}
\thanks{
The first author would like to thank North Carolina State University and the University of S\~ao Paulo for their support and hospitality during his numerous visits to Raleigh and S\~ao Paulo and Fapesp for financial support (2012/22965-5).  The second author was partially supported by Fapesp (2010/50347-9) and CNPq (301743/2007-0). The third author was partially supported by the NSA grant H98230-12-1-0248.}


\maketitle
%\subjclass{17B37, 17B67, 17B10}

\section{Introduction}

Let $\widehat{\mathfrak{g}}$ be an affine Lie algebra and $\Delta$ denote the set of roots with respect to the Cartan subalgebra $\widehat{\mathfrak{h}}$. Then we have a natural (standard) partition of $\Delta = \Delta_+ \cup \Delta_-$ into set of positive and negative roots. With respect to this standard partition we have a standard Borel subalgebra from which we may induce the standard Verma module. A partition 
$\Delta = S \cup -S$ of the root system $\Delta$ is said to be a closed partition if whenever $\alpha , \beta \in S$ and $\alpha + \beta \in \Delta$ we have $\alpha + \beta \in S$. It is well known that for any finite dimensional complex simple Lie algebra, all closed partitions of the root system are Weyl group conjugate to the standard partition. However, this is not the case for affine Lie algebras. The classification of closed subsets of the root system for affine Lie algebras was obtained by Jakobsen and Kac \cite{JK,MR89m:17032}, and independently by Futorny \cite{MR1078876,MR1175820}. In fact for affine Lie algebras there exists a finite number ($\geq 2$) of inequivalent Weyl orbits of closed partitions. Corresponding to each such non-standard partitions we have non-standard Borel subalgebras from which we can induce other non-standard Verma-type modules and these typically contain both finite and infinite dimensional weight spaces. The imaginary Verma module \cite{MR95a:17030} is a non-standard Verma-type module associated with the simplest non-standard partition of the root system 
$\Delta$ which is the focus of our study in this paper.

For generic $q$, the quantum affine algebra $U_q(\widehat{\mathfrak{g}})$ is the $q$-deformations of the universal enveloping algebras of
$\widehat{\mathfrak{g}}$ (\cite{MR802128}, \cite{MR797001}). It is known \cite{MR954661} that integrable highest weight modules of 
$\widehat{\mathfrak{g}}$ can be deformed to those over $U_q(\widehat{\mathfrak{g}})$ in such a way that the dimensions of the weight spaces are invariant under the deformation. Following the framework of \cite{MR954661} and \cite{MR1341758}, {\it quantum imaginary Verma modules} for  $U_q(\widehat{\mathfrak{g}})$ were constructed in (\cite {MR97k:17014}, \cite{MR1662112}) and it was shown that these modules are deformations of those over the universal enveloping algebra $U(\widehat{\mathfrak{g}})$ in such a way that the weight multiplicities, both finite and infinite-dimensional, are preserved.

Lusztig \cite{MR1035415} from a geometric view point and Kashiwara \cite{MR1115118} from an algebraic view point 
introduced the notion of canonical bases (equivalently, global crystal bases) for standard Verma modules
$V_q(\lambda)$ and integrable highest weight modules $L_q(\lambda)$. The crystal base (\cite{MR1090425, MR1115118}) can be thought of as the $q=0$ limit of the global crystal base or canonical base. An important ingredient in the construction of crystal base by Kashiwara in
 \cite{MR1115118}, is a subalgebra $\mathcal {B}_q$ of the quantum group which acts on the negative part of the quantum group  
 by left multiplication. This subalgebra $\mathcal {B}_q$, which we call the Kashiwara algebra, played an important role in the definition of the Kashiwara operators which defines the crystal base.

In this paper we construct an analog of Kashiwara algebra $\mathcal K_q$ for the imaginary Verma module $M_q(\lambda)$ for the quantum affine algebra $U_q(\widehat{\mathfrak{g}})$ by introducing certain Kashiwara-type operators.
Then we prove that certain quotient $\mathcal N_q^-$ of  $U_q(\widehat{\mathfrak{g}})$ is a simple $\mathcal K_q$-module.
This generalizes the corresponding result in \cite{CFM10} for the quantum affine algebra $U_q(\widehat{sl(2)})$. However, it is worth pointing out that some of the arguments involving explicit calculations in \cite{CFM10} do not extend to this general case.

The paper is organized as follows.
In Sections 2 and 3 we recall necessary definitions and some results that we need. In Section 4 we recall some facts about the imaginary Verma modules for $U_q(\widehat{\mathfrak{g}})$. In particular, for any dominant weight $\lambda$ with $\lambda (c) = 0$ we give a necessary and sufficient condition for the reduced imaginary Verma module $\tilde{M}_q(\lambda)$ to be simple. In Section 5 we define certain operators we call $\Omega$-operators acting on certain subalgebra $\mathcal{N}_q^-$ of $\tilde{M}_q(\lambda)$ and prove generalized commutation relations among them. We define the Kashiwara algebra $\mathcal K_q$ in terms of certain Drinfeld generators and the $\Omega$-operators in Section 6 and show that  $\mathcal N_q^-$ is a left $\mathcal K_q$-module and define a symmetric invariant bilinear form on $\mathcal N_q^-$.  Finally, in Section 7 we prove that $\mathcal N_q^-$ is simple as a $\mathcal K_q$-module and that the form defined in Section 6 is nondegenerate.
                                                                                                                                                                                                                                                                                                                                                                                                                                                                                                                                                                                                                                                                                                                                                                                                                                                                                                                                                                                                                                                                                                                                                                                                                                                                                                                                                                                                                                                                                                                                                                                        \section{ The affine Lie algebra $ \widehat{\mathfrak{g}}$.}

We begin by recalling some basic facts and constructions for the affine
Kac-Moody algebra $\widehat{\mathfrak{g}}$
and its imaginary Verma modules.
See \cite{K} for Kac-Moody algebra terminology and standard notations.

\subsection{} \label{notation} 
Let $I=\{0,\dots, N\}$ and $ A=(a_{ij})_{0\leq i,j\leq N}$ be a generalized affine Cartan matrix of type 1 for an untwisted affine Kac-Moody algebra $\widehat{\mathfrak g}$.   Let $D=(d_0,\dots,d_N)$ be a diagonal matrix with relatively prime integer entries such that the matrix $DA$ is symmetric.
Then $\widehat{\mathfrak g}$ has the loop space realization
$$
\widehat{\mathfrak g}=  {\mathfrak g} \otimes {\mathbb C}[t,t^{-1}] \oplus {\mathbb C} c \oplus {\mathbb C} d, 
$$
where $\mathfrak g$ is the finite dimensional simple Lie algebra over $\mathbb C$ with Cartan matrix $(a_{ij})_{1\leq i,j\leq N}$, $c$ is central in $\widehat{\mathfrak g}$; $d$ is the degree derivation,
so that $[d,x \otimes  t^n] = n x \otimes  t^n$ for any $x \in { \mathfrak g}$ and
$n \in {\mathbb Z}$, and $[x \otimes  t^n, y \otimes  t^m] = [x,y] \otimes  t^{n+m} +
\delta_{n+m,0}n(x|y)c$ for all $x,y \in  {\mathfrak g}$, $n,m \in {\mathbb Z}$.  

An alternative Chevalley-Serre presentation of $\widehat{\mathfrak g}$ is given
by defining it as the Lie algebra with generators $e_i, f_i, h_i$
 ($i \in   I$) and $d$ subject to the relations
\begin{align}\label{Chevalley-Serre}
[h_i,h_j] &= 0, \qquad [d,h_i]=0,\\
[h_i, e_j] &= a_{ij}e_j, \qquad [d,e_j] = \delta_{0,j}e_j,\\
[h_i, f_j] &= -a_{ij}f_j, \qquad [d, f_j] = -\delta_{0,j}f_j,\\
[e_i, f_j] &= \delta_{ij}h_i, \\
({{\text {ad}}} e_i)^{1-a_{ij}}(e_j)&= 0, \qquad 
({{\text {ad}}} f_i)^{1-a_{ij}}(f_j) = 0, \quad i \neq j.
\end{align}
We
set $\hat{\mathfrak h}$ to be the span of $\{h_0,\dots, h_N,d\}$. 

Let $\Delta_0$ be the set of roots of $\mathfrak g$ with chosen set of positive/negative roots $\Delta_{0,\pm}$.  Let $ Q_0$ be the free abelian group with basis $\alpha_i$, $1\leq i\leq N$ which is the root lattice of $\mathfrak g$. Let $\check Q_0=\sum_i\mathbb Z\check \alpha_i$ be the coroot lattice of $\mathfrak g$.   The co-weight lattice is defined to be $\check P_0=\text{Hom}(Q_0,\mathbb Z)$ with basis $\omega_i$ defined by $\langle \omega_i,  \alpha_j\rangle =\delta_{i,j}$.  The simple reflections $s_i:\check P_0\to \check P_0$ are defined by $s_i(x)=x-\langle \alpha_i,x\rangle \check\alpha_i$.  The $s_i$ also act on $Q_0$ by $s_i(y)=y-\langle y,\check\alpha_i\rangle \alpha_i$.
The Weyl group of $\mathfrak g$ is defined as the subgroup $W_0$ of $\text{Aut}\check P_0$ generated by $s_1,\dots,s_N$.  The affine Weyl group is defined as $W=W_0\ltimes \check Q_0$.  Let $\theta$ be the longest positive root and set $s_0=(s_\theta,-\check \theta)$.  Then $W$ is generated by $s_0,\dots, s_N$.  Let $\tilde W=W_0\ltimes \check P_0=W\ltimes T$ be the generalized affine Weyl group
where $T$ is the group of Dynkin diagram automorphisms.  \color{black}

Let $\Delta$ be the root system of $\widehat{\mathfrak g}$ with positive/negative set of roots $\Delta_{\pm}$and simple roots $\Pi=\{\alpha_0,\dots, \alpha_N\}$.  Define $\delta=\alpha_0+\theta$. Extend the root lattice $Q_0$ of $ \mathfrak g$ to the affine root lattice
$Q:=  Q _0\oplus {\mathbb Z}\delta$, and extend the form $(.|.)$ to $Q$
by setting $(q|\delta)=0$ for all $q \in Q_0$ and $(\delta|\delta)=0$. The generalized affine Weyl group $\tilde W$ acts on $Q$ as an affine transformation group.  In particular if $z\in \check P_0$ and $1\leq i\leq N$, then
$z(\alpha_i)=\alpha_i-\langle z, \alpha_i\rangle \delta$. 
Let $Q_+ = \sum_{i \in \dot I} {\mathbb Z}_{\ge 0}\alpha_i \oplus {\mathbb Z}_{\ge 0}\delta$.

The root system $\Delta$ of $\widehat{\mathfrak g}$ is given by
$$
\Delta = \{ \alpha + n \delta\ |\ \alpha \in  \Delta_0, n \in {\mathbb Z}\}
\cup
\{k\delta\ |\ k\in {\mathbb Z}, k \neq 0\}.
$$
The roots of the form $\alpha+n\delta$, $\alpha \in   \Delta, n \in {\mathbb Z}$ are
called real roots, and those of the form $k\delta$,
$k \in {\mathbb Z}, k \neq 0$ are called imaginary roots.  We let
$\Delta^{re}$ and $\Delta^{im}$ denote the sets of real and
imaginary roots, respectively.  The set of positive real roots
of  $\widehat{\mathfrak g}$ is $\Delta_+^{re} =  \Delta_{0,+}
\cup \{\alpha + n\delta \ |\ \alpha \in  \Delta_0, n>0\}$ and the set of positive
imaginary roots is $\Delta_+^{im} = \{k\delta\ |\ k>0\}$.
The set of positive roots of $\widehat{\mathfrak g}$ is
$\Delta_+ = \Delta_+^{re} \cup \Delta_+^{im}$.  Similarly, on the negative
side, we have $\Delta_- = \Delta_-^{re}\cup \Delta_-^{im}$, where
$\Delta_-^{re} = \Delta_{0,-} \cup \{ \alpha + n\delta\ |\ \alpha \in  \Delta_0, n<0\}$
and $\Delta_-^{im} = \{k\delta\ |\ k<0\}$.
The 
weight lattice $P$ of ${\widehat{\mathfrak g}}$ is
$P = \{ \lambda \in {\widehat{\mathfrak h}}^*\ |\ \lambda(h_i) \in {\mathbb Z}, i \in I, \lambda(d) \in {\mathbb Z}\}$.
Let $B$ denote the associated braid group with generators
$T_0, T_1, \dots, T_N$.

\subsection{}  \label{partition}
Consider the partition $\D = S \cup -S$ of the root system of $\g$
where $S=\{ \alpha+ n \delta\ |\ \alpha\in \D_{0,+}, n \in \mathbb Z\} \cup
\{k\delta\ |\ k>0\}$. This is a non-standard partition of the root system
$\D$ in the sense that $S$ is not Weyl equivalent to the set
$\D_+$ of positive roots.

\section{The quantum affine algebra $U_q(\widehat{\mathfrak g})$}

\subsection{}  \label{jimbodrinfeld}
The {\it quantum affine algebra}
$U_q(\widehat{\mathfrak g})$ is the $\mathbb C(q^{1/2})$-algebra with 1 generated by
$$ 
E_i, \enspace F_i, \enspace K_\alpha,\enspace \gamma^{\pm 1/2}, \enspace D^{\pm 1} \quad 0\leq i\leq N,\enspace \alpha\in Q,
$$
and defining relations:
\begin{align*}& DD^{-1}=D^{-1}D=K_iK_i^{-1}=K_i^{-1}K_i=\gamma^{1/2}\gamma^{-1/2}=1, \\
&[\gamma^{\pm 1/2},U_q(\mathfrak g)]=[D,K_i^{\pm1}]=[K_i,K_j]=0, \\
&(\gamma^{\pm 1/2})^2=K_\delta^{\pm 1},\\
& E_iF_j-F_jE_i = \delta_{ij}\frac{K_i-K_i^{-1}}{q_i-q_i^{-1}}, \\
& K_\alpha E_iK_\alpha^{-1}=q^{(\alpha|\alpha_i)}E_i, \ \ K_\alpha F_i K_\alpha^{-1} =q^{-(\alpha|\alpha_i)}F_i, \\
& DE_iD^{-1}=q^{\delta_{i,0}} E_i,\quad DF_iD^{-1}=q^{-\delta_{i,0}} F_i, \\
& \sum_{s=0}^{1-a_{ij}}(-1)^sE_i^{(1-a_{ij}-s)} E_jE_i^{(s)}=0= \sum_{s=0}^{1-a_{ij}}(-1)^sF_i^{(1-a_{ij}-s)} F_jF_i^{(s)}, \quad i\neq j.
\end{align*}
where
$$
q_i:=q^{d_i},\quad 
[n]_i= \frac{q^n_i-q^{-n}_i}{q_i-q^{-1}_i},\quad [n]_i!:=\prod_{k=1}^n[k]_i
$$
and 
$K_i=K_{\alpha_i}$, $E_i^{(s)}=E_i/[s]_i!$ and $F_i^{(s)}=F_i/[s]_i!$ (see \cite{Bec94a} and \cite{MR954661}). 

The quantum affine algebra
$U_q(\widehat{\mathfrak g})$  is a Hopf algebra with a comultiplication given by
\begin{align}
\Delta(K_i^{\pm 1}) &= K_i^{\pm 1} \otimes K_i^{\pm 1},\label{ki} \\
\Delta(D^{\pm 1})&=D^{\pm 1}\otimes D^{\pm 1},\qquad \Delta(\gamma^{\pm 1/2})=\gamma^{\pm 1/2}\otimes \gamma^{\pm 1/2}\label{dgamma}\\
\Delta(E_i) &= E_i\otimes 1 + K_i\otimes E_i, \label{ei}\\
\Delta(F_i) &= F_i\otimes K_i^{-1} + 1 \otimes F_i,\label{fi} 
\end{align}
and an antipode given by
\begin{gather*}
s(E_i) =-E_iK_i^{-1},  \quad 
s(F_i) = -K_iF_i, \\
s(K_i) = K_i^{-1}, \quad 
s(D)  = D^{-1}, \quad 
s(\gamma^{1/2}) = \gamma^{-1/2}.
\end{gather*}

Let $\Phi:U_q(\widehat{\mathfrak g})\to U_q(\widehat{\mathfrak g})$ be the $\mathbb C $-algebra automorphism defined by
\begin{gather}
\Phi(E_i)=F_i,\enspace \Phi(F_i) =E_i,\enspace \Phi(K_\alpha)=K_{\alpha},\label{automorphism} \\
 \Phi(D)=D,\enspace \Phi(\gamma^{\pm1/2})=\gamma^{\pm 1/2},\enspace \Phi(q^{\pm 1/2})=q^{\mp 1/2},\notag
\end{gather}
and let $\bar\Omega:U_q(\widehat{\mathfrak g})\to U_q(\widehat{\mathfrak g})$ be the $\mathbb C $-algebra anti-automorphism defined by
\begin{gather}\label{antiautomophism}
\bar\Omega(E_i)=F_i,\enspace \bar\Omega(F_i)=E_i,\enspace \bar\Omega(K_\alpha)=K_{-\alpha},\\ \bar\Omega(D)=D^{-1},\enspace \bar\Omega(\gamma^{\pm1/2})=\gamma^{\mp 1/2},\enspace \bar\Omega(q^{\pm 1/2})=q^{\mp1/2},\notag
\end{gather}
(see \cite[Section 1]{Bec94a}).
%\color{blue}
%The anti-automorphism $\bar\Omega$ should not be confused with the operators $\bar\Omega_{\psi_i}$ and $\bar\Omega_{\phi_i}$ given below in \eqnref{definingomegapsi}.
%
%\color{black}

\subsection{}\label{secondrealization}  There is an alternative realization for $U_q(\widehat{\mathfrak{g}})$,
due to Drinfeld
\cite{MR802128}, which we shall also need.   We will use the formulation due to J. Beck \cite{Bec94a}.   Let
$U_q(\widehat{\mathfrak g})$ be the associative algebra with $1$ over $\mathbb C(q^{1/2})$-
generated by
$$  
x_{ir}^{\pm 1},\enspace h_{is}, \enspace K_i^{\pm 1}, \enspace \gamma^{\pm 1/2},D^{\pm 1} \enspace 1\leq i\leq N,r,s\in\mathbb Z, s\neq 0,
$$
with defining relations:
\begin{align}
 DD^{-1}&=D^{-1}D=K_iK_i^{-1}=K_i^{-1}K_i=\gamma^{1/2}\gamma^{-1/2}=1,\label{drinfeldfirst} \\
[\gamma^{\pm 1/2},U_q(\mathfrak g)]&=[D,K_i^{\pm 1}]=[K_i,K_j]=[K_i,h_{jk}]=0, \\
Dh_{ir}D^{-1}&=q^rh_{ir},\quad Dx_{ir}^{\pm}D^{-1}=q^rx_{ir}^{\pm},\\
K_ix_{jr}^{\pm}K_i^{-1} &= q_i^{\pm  (\alpha_i|\alpha_j)}x_{jr}^{\pm},   \\  
[h_{ik},h_{jl}]&=\delta_{k,-l} \frac{1}{k}[ka_{ij}]_i\frac{\gamma^k-\gamma^{-k}}{q_j-q_j^{-1}}\label{hs} \\
[h_{ik},x^{\pm}_{jl}]&= \pm \frac{1}{k}[ka_{ij}]_i\gamma^{\mp |k|/2}x^{\pm}_{j,k+l}, \label{axcommutator}  \\
    x^{\pm}_{i,k+1}x^{\pm}_{jl} &- q^{\pm (\alpha_i|\alpha_j)}
    %%%%%%%%%%%%%%%%%%%%%
x^{\pm}_{jl}x^{\pm}_{i,k+1}\label{Serre}   \\
&= q^{\pm (\alpha_i|\alpha_j)}x^{\pm}_{ik}x^{\pm}_{j,l+1}
    - x^{\pm}_{j,l+1}x^{\pm}_{ik},\notag \\
[x^+_{ik},x^-_{jl}]&=\delta_{ij}
    \frac{1}{q_i-q^{-1}_i}\left( \gamma^{\frac{k-l}{2}}\psi_{i,k+l} -
    \gamma^{\frac{l-k}{2}}\phi_{i,k+l}\right), \label{xcommutator}   \\
\text{where  }
\sum_{k=0}^{\infty}\psi_{ik}z^{k} &= K_i \exp\left(
(q_i-q^{-1}_i)\sum_{l>0}  h_{il}z^{l}\right), \text{ and }\notag\\
\sum_{k=0}^{\infty}
\phi_{i,-k}z^{-k}&= K^{-1}_i \exp\left( - (q_i-q^{-1}_i)\sum_{l>0} 
h_{i,-l}z^{-l}\right).\label{phidef}\\
\text{For }i\neq j,\enspace n:=1-a_{ij} \notag  \\
\text{Sym}_{k_1,k_2,\dots,k_n}&\sum_{r=0}^{n}(-1)^r \genfrac{[}{]}{0pt}{}{n}{r} x_{ik_1}^\pm \cdots x_{ik_r}^\pm x_{jl}^\pm x_{ik_{r+1}}^\pm \cdots x_{ik_s}^\pm=0.\label{drinfeldlast}
\end{align}
 Note that Beck's paper \cite{Bec94a} on page 565 has a typo in it where he has $\phi_{i,k}z^k$ instead of $\phi_{i,-k}z^{-k}$.

In the above last relation $\text{Sym}$ means symmetrization with respect to the indices $k_1,\dots, k_n$.  Also in Drinfeld's notation one has $e^{hc/2}=\gamma$ and $e^{h/2}=q$.

The algebras given above and in \secref{jimbodrinfeld} are isomorphic \cite{MR802128}. 
If one uses the formal sums
\begin{equation}
\phi_i(u)=\sum_{p\in\mathbb Z} \phi_{ip}u^{-p},\enspace \psi_i(u)=\sum_{p\in\mathbb Z}\psi_{ip}u^{-p},\enspace
x_i^{\pm }(u)=\sum_{p\in\mathbb Z} x_{ip}^\pm u^{-p}
\end{equation}
Drinfeld's relations \eqnref{hs}-\eqnref{xcommutator} can be written as
\begin{gather}
[\phi_i(u),\phi_j(v)]=0=[\psi_i(u),\psi_j(v)] \\
\phi_i(u)\psi_j(v)\phi_i(u)^{-1}\psi_j(v)^{-1}=g_{ij}(uv^{-1}\gamma^{-1})/g_{ij}(uv^{-1}\gamma)  \\ 
\phi_i(u)x^\pm_j (v)\phi_i(u)^{-1}=g_{ij}(uv^{-1}\gamma^{\mp 1/2})^{\pm 1}x_j^\pm (v)\label{phix} \\
\psi_i(u)x_j^\pm (v)\psi_i(u)^{-1}=g_{ji}(vu^{-1}\gamma^{\mp 1/2})^{\mp 1}x_j^\pm (v)\label{psix} \\
(u-q^{\pm(\alpha_i|\alpha_j)} v)x^\pm_i (u)x^\pm_j (v)=(q^{\pm (\alpha_i|\alpha_j)}u-v)x_j^\pm(v)x_i^\pm(u) \\
[x_i^+(u),x_j^-(v)]=\delta_{ij}(q_i-q^{-1}_i)^{-1}(\delta(u/v\gamma)\psi_i(v\gamma^{1/2})-\delta(u\gamma/v)\phi_i(u\gamma^{1/2}))\label{xx}
\end{gather}
where $g_{ij}(t)=g_{ij,q}(t)$ is the Taylor series at $t=0$ of the function $(q^{(\alpha_i|\alpha_j)}t-1)/(t-q^{(\alpha_i|\alpha_j)})$ and $\delta(z)=\sum_{k\in\mathbb Z}z^{k}$ is the formal Dirac delta function.

\subsection{}\label{drinfeldgen}

Let $U_q^+=U_q^+(\g)$ (resp. $U_q^-=U_q^-(\g)$) be the subalgebra of
$U_q(\g)$ generated by $E_i$ (resp. $F_i$), $i \in I$, and
let $U_q^0=U_q^0(\g)$ denote the subalgebra generated by
$K_i^{\pm 1}$ ($i \in I$) and $D^{\pm 1}$.

Beck in \cite{Bec94a} and \cite{Bec94b} has given a total ordering of the
root system $\D$ and a PBW like basis for $U_q(\g)$.
Below we follow the construction developed by Damiani \cite{Dam98}, Gavarini \cite{Gav99} and \cite{BK96} and
let $E_{\beta}$ denote the root vectors for each
$\beta \in \Delta_+$ counting with multiplicity for the
imaginary roots.  One defines $F_\beta=E_{-\beta}:=\bar\Omega(E_\beta)$ for $\beta\in\Delta_+$ (refer to \eqnref{antiautomophism}).

For any affine Lie algebra $\widehat{\mathfrak g}$, there exists a map $\pi :\mathbb Z \to I$
such that, if we define
$$
\beta_k =
\begin{cases} 
&s_{\pi(0)}s_{\pi(-1)}\cdots s_{\pi(k+1)}(\alpha_{\pi(k)})
\qquad \text{ for all } k < 0, \\
&\alpha_{\pi(0)}\hskip 130pt k=0, \\
& \alpha_{\pi(1)}\hskip 130pt k=1, \\
&s_{\pi(1)}s_{\pi(2)} \cdots s_{\pi(k-1)}(\alpha_{\pi(k)})
\qquad  \text{ for all } k> 1,
\end{cases}
$$
then the map $\pi':\mathbb Z \mapsto \Delta_+^{re}$ given by $\pi'(k)=\beta_k$
is a bijection. Note that the map $\pi$, and hence the total ordering, is not unique.
We fix $\pi$ so that
$ \{ \beta_k\ |\ k \le 0\} =
\{ \alpha + n\delta \ |\ \alpha \in  \Delta_{0,+}, n \ge 0\}$ and
$\{ \beta_k\ |\ k \ge 1\} =
\{ -\alpha + n\delta\ |\ \alpha \in  \Delta_{0,+}, n>0\}$.  
One also defines the set of imaginary roots with multiplicity as
%subsets of $\Delta$ by
\begin{gather*}
%\Delta_+(k):=\begin{cases} \{\beta_r\,|\,1\leq r< k\},\quad \text{ if }k\geq 1, \\
% \{\beta_r\,|\,0\leq r> k\},\quad \text{ if }k\leq 0; \\
%\end{cases}  \\
%\Delta_+(\infty):=\left\{\beta_r\,|\, r\geq 1\right\},\quad \Delta_+(-\infty):=\left\{\beta_r\,|\, r\leq 0\right\} , \\
\Delta_+(\text{im}):=\Delta_+^{\text{im}}\times  I_0,
%\Delta_+(\tilde\infty):=\Delta_+(\infty)\cup \Delta_+(\text{im}),\quad \Delta_+(-\tilde\infty):=\Delta_+(-\infty)\cup \Delta_+(\text{im}).
\end{gather*}
where  $I_0=\{1,...,N\}$.
%\color{red} Is the above needed for what follows?\color{black}

It will be convenient for us to invert Beck's original ordering
of the positive roots (see \cite[\S 1.4.1]{BK96}).   Let
\begin{equation}
\beta_0 > \beta_{-1} > \beta_{-2}> \dots >\delta >2\delta>  \dots
>\beta_2 >\beta_1,
\end{equation}
(see \cite[\S 2.1]{Gav99} for this ordering).
We define $-\alpha < -\beta$ iff $\beta > \alpha$ for all
positive roots $\alpha, \beta$, so we obtain a corresponding ordering
on $\Delta_-$.

The following elementary observation on the ordering will play
a crucial role later.  Write $A<B$ for two sets $A$ and $B$ if
$x < y$ for all $x \in A$ and $y \in B$.  Then Beck's total
ordering of the positive roots can be divided into three sets:
$$
\{ \alpha + n\delta\ |\ \alpha \in  \Delta_{0,+}, n \ge 0\} >
\{k\delta \ |\ k>0\} > \{-\alpha+k\delta\ |\ \alpha \in   \Delta_{0,+}, k>0\}.
$$
Similarly, for the negative roots, we have,
$$
\{ -\alpha - n\delta\ |\ \alpha \in \Delta_{0,+}, n \ge 0\} <
\{-k\delta \ |\ k>0\} < \{\alpha-k\delta\ |\ \alpha \in   \Delta_{0,+}, k>0\}.
$$

The action of the braid group generators $T_i$ on the generators
of the quantum group $U_q(\widehat{\mathfrak g})$ is given by the following.
\begin{align*}
T_i(E_i) &= -F_i K_i, \qquad T_i(F_i) = -K_i^{-1}E_i, \\
T_i(E_j) &= \sum_{r=0}^{-a_{ij}} (-1)^{r-a_{ij}}
q_i^{-r}
E_i^{(-a_{ij}-r)}E_jE_i^{(r)}, \qquad \text{if } i \neq j,\\
T_i(F_j) &=  \sum_{r=0}^{-a_{ij}} (-1)^{r-a_{ij}}q_i^r
F_i^{(r)}F_jF_i^{(-a_{ij}-r)}, \qquad \text{if } i \neq j,\\
T_i(K_j) &= K_jK_i^{-a_{ij}}, \qquad
T_i(K_j^{-1}) = K_j^{-1}K_i^{a_{ij}}, \\
T_i(D) &= DK_i^{-\delta_{i,0}}, \qquad
T_i(D^{-1}) = D^{-1}K_i^{\delta_{i,0}}.
\end{align*}

For each $\beta_k \in \Delta_+^{re}$, define the root vector
$E_{\beta_k}$ in $U_q(\g)$ by
\begin{align}
E_{\beta_k} =
\begin{cases}
&T^{-1}_{\pi(0)}T^{-1}_{\pi(-1)}\cdots T^{-1}_{\pi(k+1)}(E_{\pi(k)})
\quad \text{ for all } k < 0, \\
&E_{\pi(0)}  \hskip 120ptk=0,\\
&E_{\pi(1)} \hskip 120pt k = 1, \\
&T_{\pi(1)}T_{\pi(2)} \cdots T_{\pi(k-1)}(E_{\pi(k)})
\qquad  \text{ for all } k>1.
\end{cases}\label{realrootvector}
\end{align}
 Orient the Dynkin diagram of $\mathfrak g$ by defining a map $o:V\to \{\pm 1\}$ so that for adjacent vertices $i$ and $j$ one has $o(i)=-o(j)$.   Beck defines  $\widehat{T}_{\omega_i}=o(i)T_{\omega_i}$ 
and obtains (\cite[Section 4]{Bec94a}) for $i\in \dot I$ and $k\in\mathbb Z$, 
\begin{align*}
x_{ik}^-:&=\widehat{T}_{\omega_i}^k(F_i), \enspace x_{ik}^+:=\widehat{T}_{\omega_i}^{-k}(E_i).
\end{align*}

The following result is due to Iwahori, Matsumoto and Tits (see \cite{Bec94a}, Section 2).
\begin{prop}
Suppose $w\in\tilde W$ and $w=\tau s_{i_1}\cdots s_{i_n}$ is a reduced decomposition in terms of simple reflections.   Then $T_w=\tau T_{i_1}\cdots T_{i_n}$ does not depend on the reduced decomposition of $w$ chosen, but rather only on $w$. 
\end{prop}

Fix $i\in I_0$ and $k\geq 0$. The proposition above in the particular case of the reduced decomposition of $\omega_i=\tau s_{i_1}\cdots s_{i_r}\in \check P_0\subset \tilde W$ where $\tau$ is a diagram automorphism and the $s_i$ are simple reflections, gives
$$
x_{ik}^+=\widehat{T}_{\omega_i}^{-k}(E_i)=o(i)^k( \tau T_{i_1}\cdots T_{i_r})^{-k}(E_i)=o(i)^k  T_{j_1}\cdots T_{j_m}\tau^{-k}(E_i),
$$
for some $j_t\in I$. 

Fixing still $i\in I_0$ and $k\geq 0$, choose now $w_{\alpha_i+k\delta}\in \tilde W$, and $j\in I$, such that $w_{\alpha_i+k\delta}(\alpha_j)=\alpha_i+k\delta$.  Writing $w_{\alpha_i+k\delta}=s_{l_1}\cdots s_{l_p}$ as a reduced decomposition of simple reflections,  Beck defines 
\begin{align*}
E_{\alpha_i+k\delta}:&=T_{w_{\alpha_i+k\delta}}(E_j)=T_{l_1}\cdots T_{l_p}(E_j),
\end{align*}
which according to Lusztig is independent of the choice of $w_{\alpha_i+k\delta}$, its reduced decomposition and $j\in I$.  In particular we can choose $j=\tau^{-k}(i)$ and $w=s_{j_1}\cdots s_{j_m}$, so that $s_{j_1}\cdots s_{j_m}(\alpha_j)=s_{j_1}\cdots s_{j_m}(\alpha_{\tau^{-k}(i)})=\alpha_i+k\delta$.  Then 
\begin{equation}
E_{\alpha_i+k\delta}=T_{j_1}\cdots T_{j_m}(E_{\tau^{-k}(i)})=o(i)^k x^+_{ik}.\label{rootvector1}
\end{equation}
Now one defines
\begin{align}\label{rootvector2}
F_{\alpha_i+k\delta}%=E_{-\alpha_i-k\delta}
&=\bar\Omega(E_{\alpha_i+k\delta})=o(i)^k  \bar\Omega(x^+_{ik})=o(i)^k  \bar\Omega(\widehat{T}_{\omega_i}^{-k}(E_i))  \\
&=o(i)^k  \widehat{T}_{\omega_i}^{-k}(\bar\Omega(E_i))=o(i)^k  \widehat{T}_{\omega_i}^{-k}(F_i) =o(i)^k  x_{i,-k}^-,\notag
\end{align}
as $T_j\bar\Omega=\bar\Omega T_j$ and $T_\tau \bar\Omega=\bar\Omega T_\tau$.

If $k<0$ and $i\in I_0$, then $-\alpha_i-k\delta \in \Delta_+^{\text{re}}$, so that $-\alpha_i-k\delta=\beta_l=s_{\pi(1)}\cdots s_{\pi(l-1)}(\alpha_{\pi(l)})$ for $l> 1$ and $-\alpha_i-k\delta=\beta_l= \alpha_{\pi(1)}$ if $l=1$. Then for $l>1$,
\begin{align}
E_{-\alpha_i-k\delta}&=E_{\beta_l}=T_{\omega_i}^{-k}T_i^{-1}(E_i)=-T_{\omega_i}^{-k}(K_i^{-1}F_i) \label{rootvector3}
\\
&=-o(i)^kT_{\omega_i}^{-k}(K_i^{-1})x_{i,-k}^-=-o(i)^k K_i^{-1}\gamma^{-k}x_{i,-k}^-\notag
\end{align}
as $\omega_i(-\alpha_i)=-\alpha_i+\delta$ (see \secref{notation}) so that $\omega_i^{-k}s_i(\alpha_i)=\omega_i^{-k}(-\alpha_i)=-\alpha_i-k\delta$ and $T_{\omega_i}(K_i^{-1})=K_{-\alpha_i+\delta}$.  Now 
\begin{align}\label{rootvector4}
F_{-\alpha_i-k\delta}=\bar\Omega(E_{-\alpha_i-k\delta})
&=-o(i)^k \bar\Omega(K_i^{-1}\gamma^{-k}x_{i,-k}^-)\\ 
&=-o(i)^k K_i \gamma^{k} x_{i,k}^+.\notag
\end{align}

Then as shown in \cite[Theorem 4.7]{Bec94a}, for $k>0$
\begin{align*} 
\psi_{ik}&=(q_i-q_i^{-1})\gamma^{k/2}[E_i,\hat T_{\omega_i}^k(F_i)] =(q_i-q_i^{-1})\gamma^{k/2}[E_i, x_{ik}^-], \\
\phi_{i,-k}&=(q_i-q_i^{-1})\gamma^{-k/2}[F_i,\hat T_{\omega_i}^kE_i]=(q_i-q_i^{-1})\gamma^{-k/2}[F_i, x_{i,-k}^+],
\end{align*}
$\psi_{i,0}:=K_i$, $\phi_{i,0}:=K_i^{-1}$, 
and for any $\tau \in \mathcal T$, 
\begin{equation}
T_\tau(E_i):=E_{\tau(i)}, \quad 
T_\tau(F_i):=F_{\tau(i)},\quad T_\tau(K_i):=K_{\tau(i)}\label{tau}.
\end{equation}
One writes $\tau$ for $T_\tau$.
Note also that $\tau s_i\tau^{-1}=s_{\tau(i)}$ for all $0\leq i\leq n$.  

Each real root space is 1-dimensional, but each imaginary root
space is $N$-dimensional.  Hence, for each positive imaginary
root $k\delta$ ($k>0$) one defines the $N$ imaginary root vectors,
$E_{k \delta}^{(i)}$ ($i\in  I_0$) by
\begin{align}
\exp\left(
(q_i-q^{-1}_i)\sum_{k=1}^{\infty}E_{k \delta}^{(i)}z^k\right) &=
1 + (q_i-q^{-1}_i)\sum_{k=1}^{\infty} K_i^{-1}[E_i, x_{i,k}^-]z^k\label{imaginaryrootvectordef}
 \\
&=1 +  \sum_{k=1}^{\infty} K_i^{-1} \psi_{ik}\left(\gamma^{-1/2}z\right)^k  \notag  \\
&=  \exp\left( (q_i-q^{-1}_i)\sum_{l>0}  h_{il}\gamma^{-l/2}z^{l}\right). \notag 
 \end{align}
 So $E_{k\delta}^{(i)}=h_{ik}\gamma^{-k/2}$ for all $k>  0$.  For $k<0$ we also define $E_{k\delta}^{(i)}:=\bar\Omega(E_{-k\delta}^{(i)})=h_{ik}\gamma^{k/2}$. Our definition of $E_{k\delta}^{(i)}$ is the same as  \cite{Dam98}, Definition 7.  
Recall that the $R$- ``matrices" are defined having values in $U_q(\widehat{\mathfrak g})\widehat{\otimes} U_q(\widehat{\mathfrak g})$  (see \cite{Lus93} for the definition of $U_q(\widehat{\mathfrak g})\widehat{\otimes} U_q(\widehat{\mathfrak g})$ and  \cite[Section 5]{Bec94a}) for $1\leq i\leq N$ by 
\begin{align}
R_i&=\sum_{n\geq 0}(-1)^nq_i^{\frac{-n(n-1)}{2}}(q_i-q_i^{-1})^n[n]_i!T_i(F_i^{(n)})\otimes T_i(E_i^{(n)}),  \label{ri}\\ 
&=\sum_{n\geq 0} \frac{(q_i^{-1}-q_i)^nq_i^{\frac{-n(5n-1)}{2}}}{[n]_i!}E_i^{n}K_i^{-n}\otimes F_i^{n}K_i^n, \notag\\ \notag  \\
\bar{R}_i&=T_i^{-1}\otimes T_i^{-1}\circ R_i^{-1}=\sum_{n\geq 0}q_i^{\frac{n(n-1)}{2}}(q_i-q_i^{-1})^n[n]_i!F_i^{(n)}\otimes E_i^{(n)}.\label{barri}
\end{align}
These operators have inverses
\begin{align*}
R_i^{-1}
&=\sum_{n\geq 0} \frac{(q_i-q_i^{-1})^nq_i^{\frac{-n(3n+1)}{2}}}{[n]_i!}E_i^{n}K_i^{-n}\otimes F_i^{n}K_i^n \\  \\
\bar{R}_i^{-1}&=\sum_{n\geq 0}\frac{q_i^{\frac{-n(n-1)}{2}}(q_i^{-1}-q_i)^n}{[n]_i!}F_i^{n}\otimes E_i^{n}
\end{align*}

Suppose $w\in \tilde W$ and $\tau s_{i_1}\cdots s_{i_r}$ is a reduced presentation for $w$ where $\tau$ is defined as in \eqnref{tau}.   Beck defines the following ``$R$-matrices":
\begin{align}
R_w&=\tau (S_{i_1}S_{i_2}\cdots S_{i_{r-1}}(R_{i_r})\cdots S_{i_1}(R_{i_2})R_{i_1}),\label{rw}\\ 
\bar R_w&=\tau (S_{i_r}^{-1}S_{i_{r-1}}^{-1}\cdots S_{i_{2}}^{-1}(R_{i_1})\cdots S_{i_r}^{-1}(\bar R_{i_{r-1}})\bar R_{i_r}).  \label{barrw}
\end{align}

\color{black}
Using the root partition $ S = \{\alpha + k\delta  \ |\alpha\in \Delta_{0,+},\ k \in \mathbb Z \} \cup
\{ l\delta \ |\ l \in \mathbb Z_{>0}\}$ from Section 2.3, we define:
 \vskip 5pt
$U_q^+(S)$ to be the subalgebra of $U_q(\widehat{\mathfrak g})$ generated by $x^+_{i,k}$,  $(1\leq i\leq N, k\in \mathbb Z)$ and
$h_{i,l}$ $(1\leq i\leq N,\ l>0)$;
 \vskip 5pt
$U_q^-(S)$ to be the subalgebra of $U_q(\widehat{\mathfrak g})$ generated by $x^-_{i,k}$
$(1\leq i\leq N, k\in \mathbb Z)$ and
$h_{i,-l}$ $(1\leq i\leq N,\ l>0)$, and
\vskip 5pt
$U_q^0(S)$ to be the subalgebra of $U_q(\widehat{\mathfrak g})$ generated by $K^{\pm 1}_i$  $(1\leq i\leq N)$,
$\gamma^{\pm 1/2}$, and $D^{\pm 1}$.   Thus $U_q^0(S)=U_q^0(\hat{\mathfrak g})$.

\subsection{}

Let $\omega$ denote the standard $\mathbb C(q^{1/2})$-linear antiautomorphism of
$U_q(\widehat{\mathfrak g})$, and set $E_{-\alpha} = \omega(E_{\alpha})$ for all $\alpha \in \Delta_+$.
Then $U_q$ has a basis of elements of the form $E_-HE_+$,
where $E_{\pm}$ are ordered monomials in the $E_{\alpha}$,
$\alpha\in \Delta_{\pm}$,  and $H$ is a monomial in $K_i^{\pm 1}$, $\gamma^{\pm 1/2}$,
and $D^{\pm 1}$ (which all commute).

Furthermore, this basis is, in Beck's terminology, convex, meaning
that, if $\alpha, \beta \in \Delta_+$ and $ \beta >  \alpha$, then
\begin{equation}\label{convexity}
E_{\beta}E_{\alpha} - q^{(\alpha|\beta)}E_{\alpha}E_{\beta} =
\sum_{\alpha <\gamma_1<\dots <\gamma_r <\beta}c_{\gamma}
E_{\gamma_1}^{a_1}\cdots E_{\gamma_r}^{a_r}
\end{equation}
for some integers $a_1,\dots, a_r$ and scalars
$c_{\gamma}\in \mathbb C[q, q^{-1}]$,
$\gamma = (\gamma_1,\dots,\gamma_r)$ (see \cite[Proposition 1.7c]{BK96}, \cite{LS90}), and
similarly for the negative roots.  The above is called the {\it Levendorski and So{\u\i}belman's convexity formula}.

Set $\mathbb A= \mathbb C[q^{1/2},q^{-1/2}, \frac{1}{[n]_{q_i}}, i\in I, n>1]$.   We first begin with a slightly different $\mathbb A$-form than in \cite{MR1662112}. Namely we define this algebra $U_{\mathbb A}=U_{\mathbb A}(\widehat{\mathfrak g})$ to be the $\mathbb A$-subalgebra
of $U_q(\widehat{\mathfrak g})$ with 1 generated by the elements 
$$  
x_{ir}^{\pm 1},\enspace h_{is}, \enspace K_i^{\pm 1}, \enspace \gamma^{\pm 1/2},D^{\pm 1} ,  \Lnum{K_i}{s}{n},\Lnum{D}{s}{n},\Lnum{\gamma}{s}{1}, \Lnum{\gamma\psi_i}{k,l}{1}
$$
for $1\leq i\leq N,r,s\in\mathbb Z,s\neq 0$ where following \cite{MR954661}, for each $i \in I$, $s \in \mathbb Z$ and $n \in  \mathbb Z_+$,
we define the {\it Lusztig elements} in $U_q(\g)$:
\begin{align}
\Lnum{\gamma}{s}{1}_i&= \frac{\gamma^{s}-\gamma^{-s}}{q_i-q^{-1}_i},\\
\Lnum{\gamma\psi_i}{k,l}{1}&= \frac{\gamma^{\frac{k-l}{2}}\psi_{i,k+l} -
    \gamma^{\frac{l-k}{2}}\phi_{i,k+l}}{q_i-q^{-1}_i} \\ 
\Lnum{K_i}{s}{n} &=
\prod_{r=1}^n
\frac{K_iq_i^{s-r+1} - K_i^{-1}q_i^{-(s-r+1)}}{q_i^r-q_i^{-r}}, \quad \text{and}\\
\Lnum{D}{s}{n} &= \prod_{r=1}^n
\frac{Dq_0^{s-r+1} - D^{-1}q_0^{-(s-r+1)}}{q_0^r-q_0^{-r}}.
\end{align}
where $q_0=q^{d_0}$. 
  This $\mathbb A$-form can be shown to be the same as that in  \cite{MR1662112} with the exception that we have added the generators $\gamma^{\pm 1/2}$, $\Lnum{\gamma}{s}{1}_i$ and $\Lnum{\gamma\psi_i}{k,l}{1}$. 
  Let $U_{\mathbb A}^+$ (resp. $U_{\mathbb A}^-$) denote the subalgebra of
$U_{\mathbb A}$ generated by the $x_{ik}^+$, $h_{il}$, where $k\in\mathbb Z$, $l\in\mathbb N\backslash\{0\}$, $1\leq i\leq N$ (resp. $x_{ik}^-$, $h_{il}$, where $k\in\mathbb Z$, $l\in-\mathbb N\backslash\{0\}$, $1\leq i\leq N$ ), $i \in I$, and
let $U_{\mathbb A}^0$ denote the subalgebra of $U_{\mathbb A}$ generated
by the elements $\gamma^{\pm 1/2},K_i^{\pm 1}, \Lnum{K_i}{s}{n}$,
$ D^{\pm 1}, \Lnum{D}{s}{n}$,
$\Lnum{\gamma}{s}{1}_i$, $ \Lnum{\gamma\psi_i}{k,-k}{1}$.
Note that if $k+l>0$ (resp. $k+l<0$), then $ \Lnum{\gamma\psi_i}{k,l}{1}\in U_{\mathbb A}^+$ (resp. $ \Lnum{\gamma\psi_i}{k,l}{1}\in U_{\mathbb A}^-$).

Let $Aut(\Gamma)$ be the set of automorphisms of the afffine Dynkin diagram $\Gamma$.  Recall $I_0=\{1,...,N\}$, and let $\pi:\mathbb Z\ni r\mapsto\pi_r\in I$, $N_1,...,N_n\in\mathbb N$, $\tau_1,...,\tau_n\in Aut(\Gamma)$ be such that: 

\begin{enumerate}[i).]
\item $N_i=\sum_{j=1}^i l(\omega_j)$ $\forall i\in I_0$ (where $\langle \omega_i,\alpha_j\rangle=\delta_{ij} $ for all $i,j\in I_0$); 

\item  $s_{\pi_1}\cdots s_{\pi_{_{N_i}}}\tau_i=\sum_{j=1}^i \omega_j$ $\forall i\in I_0$; 

\item $\pi_{r+N_n}=\tau_n(\pi_r)$ $\forall r\in\mathbb  Z$;
\end{enumerate}
\color{black}
(these conditions imply that for all $r<r^{\prime}\in\mathbb Z$,  $s_{\pi_r}s_{\pi_{r+1}}\cdots s_{\pi_{r^{\prime}-1}}s_{\pi_{r^{\prime}}}$ is a reduced expression, see \cite{IM65} and \cite{K})

Then 
$\pi$ induces a map
$$
\mathbb Z\ni r\mapsto w_r\in W\ \ {\text{
defined\ by}}\ \ w_r=\begin{cases}s_{\pi_0}\cdots  s_{\pi_{r+1}}&{\text{if}}\ r<0,
 \\   1
&{\text{if}}\ r=0,1,\\  s_{\pi_1}\cdots  s_{\pi_{r-1}}&{\text{if}}\ r> 1,
\end{cases}
$$
\color{black}
which gives the bijection
$$\mathbb Z\ni r\mapsto\beta_r=w_r(\alpha_{\pi_r})\in\Phi^{{\text{
re}}}_+.$$
Of course we also have a bijection: $\{\pm\}\times\mathbb Z\leftrightarrow\Phi^{{\text{
re}}}$.

For all $\alpha=\beta_r\in\Phi^{{\text{
re}}}_+$ as in \eqnref{realrootvector} the root vectors $E_{\alpha}$ can now be written as: 
$$E_{\beta_r}=\begin{cases}
T_{w_r^{-1}}^{-1}(E_{\pi_r})&{\text{
if}}\ r< 0, \\
E_{\pi(0)}  & \text{if } r=0,\\
E_{\pi(1)}  & \text{if }r = 1, \\
 T_{w_r}(E_{\pi_r})&{\text{
if}}\ r> 1\\\end{cases}$$ 
and we define
$$F_{\alpha}=\bar\Omega(E_{\alpha}).$$

For $r\in\mathbb Z$, we define
$$\beta_r^{\pm}=\begin{cases}
\pm\beta_r&{\text{
if}}\ r\leq 0\cr\mp\beta_r&{\text{
if}}\ r> 0;\end{cases}$$
then of course 
$$\{\beta_r^+|r\in\mathbb Z\}=\{m\delta+\alpha\in\Phi|m\in\mathbb Z,\alpha\in Q_{0,+}\},$$
$$\{\beta_r^-|r\in\mathbb Z\}=\{m\delta-\alpha\in\Phi|m\in\mathbb Z,\alpha\in Q_{0,+}\}.$$

The root vectors do depend on $\pi$ (for example if $a_{ij}=a_{ji}=-1$ we have $T_i(E_j)\neq T_j(E_i)$). What is independent of $\pi$ are the root vectors relative to the roots $ m\delta\pm\alpha_i$:
$$
E_{ m\delta+\alpha_i}=T_{\omega_i}^{-m}(E_i),\ \ E_{ m\delta-\alpha_i}=T_{\omega_i}^{m}T_i^{-1}(E_i).$$

Let $m\in\mathbb Z$, $\alpha\in Q_{0,+}$ be such that $m\delta\pm\alpha\in\Delta$; consider the following modified root vectors:

$$X_{m\delta+\alpha}=\begin{cases}
 E_{m\delta+\alpha}&{\text{
if}}\ m\geq 0,  \\
-F_{-m\delta-\alpha}K_{-m\delta-\alpha}&{\text{
if}}\ m< 0,\end{cases}
$$
$$
X_{m\delta-\alpha}=\begin{cases}
 -K_{m\delta-\alpha}^{-1}E_{m\delta-\alpha}&{\text{
if}}\ m> 0,\cr
F_{-m\delta+\alpha}&{\text{
if}}\ m\leq 0,\end{cases}$$
($\bar\Omega(X_{m\delta\pm\alpha})=X_{-m\delta\mp\alpha}$).

Equivalently
$$X_{\beta_r^+}=\begin{cases}
 E_{\beta_r}&{\text{
if}}\ r\leq 0\cr-F_{\beta_r}K_{\beta_r}&{\text{
if}}\ r\geq 1,\end{cases}\ \ \ X_{\beta_r^-}=\begin{cases}
 F_{\beta_r}&{\text{
if}}\ r\leq 0\cr-K_{\beta_r}^{-1}E_{\beta_r}&{\text{
if}}\ r\geq 1.
\end{cases}$$

\begin{thm}[\cite{CDFM13}]\label{mainresult}
Given $\um:\mathbb Z\ni r\mapsto m_r\in\mathbb N$ such that $\#\{r\in\mathbb Z|m_r\neq 0\}<\infty$ define
$$X^-(\um)=\prod_{r\in\mathbb Z}X_{\beta_r^-}^{m_r},\ \ X^+(\um)=\prod_{r\in\mathbb Z}X_{\beta_r^+}^{m_r}$$
where one chooses a fixed ordering for the products.

Given $\ul: \Delta_+(\text{\rm im})\to\mathbb N$ such that $\#\{(r\delta,i)\in \Delta_+(\text{\rm im})|l_{(r\delta,i)}\neq 0\}<\infty$ define 
$$
E^{{\text{
im}}}(\ul)=\prod_{(r\delta,i)\in  \Delta_+(\text{\rm im})}E_{(r\delta,i)}^{l_{(r\delta,i)}},\ \ 
F^{{\text{
im}}}(\ul)=\bar\Omega(E^{{\text{
im}}}(\ul)),
$$
where $E_{(r\delta,i)}=E^{(i)}_{r\delta}$.
Then the set 
\begin{equation}\label{imaginarypbw}
\{X^-(\um)F^{{\text{
im}}}(\ul)K_{\alpha}D^r\gamma^{s/2}E^{{\text{im}}}(\ul^{\prime})X^+(\um^{\prime})\},\quad r,s\in\mathbb Z,\quad \alpha\in Q_0
\end{equation}
is a basis of $U_q(\hat{\mathfrak g})$. 
\end{thm}

\section{Imaginary Verma Modules}
The algebra $\g$ has a triangular decomposition
$\g = \widehat{\mathfrak g}_{-S} \oplus \widehat{\mathfrak h}\oplus \widehat{\mathfrak g}_S$, where
$\widehat{\mathfrak g}_S = \oplus_{\alpha\in S}\widehat{\mathfrak g}_{\alpha}$ and $S$ is defined in \secref{partition}.
Let $U(\widehat{\mathfrak g}_S)$ (resp. $U(\widehat{\mathfrak g}_{-S})$) denote the universal enveloping
algebra of $\widehat{\mathfrak g}_S$ (resp. $\widehat{\mathfrak g}_{-S}$).

Let $\lambda \in P$, where $P$ is the weight lattice of $\g$.
A weight (with respect to $\widehat{\mathfrak h}$) $U(\g)$-module $V$ is
called an $S$-highest weight module with highest weight $\lambda$ if there
is some nonzero vector $v\in V_{\lambda}$ such that
\begin{enumerate}[(i).]
\item $u^+ \cdot v = 0$ for all $u^+ \in   \widehat{\mathfrak g}_{S}$;
\item  $V = U(\g)\cdot v$.
\end{enumerate}

Let $\lambda \in P$.  We make $\mathbb C$ into a 1-dimensional
$U(\widehat{\mathfrak g}_{S} \oplus\widehat{ \mathfrak h})$-module by picking a generating
vector $v$ and setting
$(x+h)\cdot v = \lambda(h)v$, for all $ x\in \widehat{\mathfrak g}_{S}, h \in \widehat{\mathfrak h}$.
The induced module
$$
M(\lambda) = U(\g) \otimes_{U(\widehat{\mathfrak g}_{S} \oplus \widehat{ \mathfrak h})}\mathbb Cv = U(\widehat{\mathfrak g}_{-S})\otimes \mathbb C v
$$
is called the {\it imaginary Verma module} with
$S$-highest weight $\lambda$.
Imaginary Verma modules are in many ways similar to ordinary
Verma modules except they contain both finite and
infinite-dimensional weight spaces.  They
were studied in \cite{MR95a:17030}, from which we
summarize.

\begin{prop}[\cite{MR95a:17030}, Proposition 1, Theorem 1]\label{irreducibilitythm}
Let $\lambda \in P$, and let $M(\lambda)$ be the imaginary Verma module
of $S$-highest weight $\lambda$.  Then $M(\lambda)$ has the following properties.
\begin{enumerate}[(i).]
\item The module $M(\lambda)$ is a free $U(\widehat{\mathfrak g}_{-S})$-module of rank 1
generated by the $S$-highest weight vector $1 \otimes 1$ of weight $\lambda$.
\item  $M(\lambda)$ has a unique maximal submodule.
\item  Let $V$ be a $U(\g)$-module generated by some
$S$-highest weight vector $v$ of weight $\lambda$.  Then there exists a
unique surjective homomorphism $\phi:M(\lambda) \mapsto V$ such
that $\phi(1\otimes  1) = v$.
\item  $\dim M(\lambda)_{\lambda} = 1$.  For any $\mu = \lambda - k\delta$, $k$
a positive integer, $0< \dim M(\lambda)_{\mu} < \infty$.  If
$\mu \neq \lambda - k \delta$ for any integer $k \ge 0$ and
$\dim M(\lambda)_{\mu} \neq 0$, then $\dim M(\lambda)_{\mu} = \infty$.
\item  The module $M(\lambda)$ is irreducible if and only if
$\lambda(c) \neq 0$.
\end{enumerate}
\end{prop}

\subsection{The Subalgebras $U_q(-S)$ and $U_q^-(S)$ of $U_q(\hat{\mathfrak g})$}
Let $U_q(\pm S)$ be the subalgebra of $U_q(\g)$ generated by
$\{ X_{\beta_r^\pm}\ | r\in\mathbb Z\}\cup\{E_{\pm k\delta}^{(i)}|\,1\leq i\leq N, k>0\}$,
and let
$B_q^{r}$ denote the subalgebra of $U_q(\g)$ generated
by $U_q(S)\cup U_q^0(\g)$ (the superscript $r$ is used to remind us that it is generated in part by root vectors).
Let $U_q^\pm( S)$ be the subalgebra of $U_q(\g)$ generated by
$\{ x^\pm_{ik}\ | 1\leq i\leq N, k\in\mathbb Z\}\cup\{h_{il}|\,1\leq i\leq N, l\in\pm\mathbb N^*\}$,
and let
$B_q^{d}$ denote the subalgebra of $U_q(\g)$ generated
by $U_q^+(S)\cup U_q^0(\g)$.  (The superscript $d$, is used to remind us that the respective subalgebras are generated in part by Drinfeld generators).

Let $\lambda \in P$.
A $U_q(\g)$ weight module $V_q^r$ is called an $S$-highest weight
module with highest weight $\lambda$ if there is a non-zero vector
$v \in V_q^r$ of weight $\lambda$  such that:
\begin{enumerate}[(i).]
\item $u^+\cdot v  =0$ for all $u^+ \in U_q(S) \setminus \mathbb C(q^{1/2})^*$;
\item $V_q^r = U_q(\g) \cdot v$.
\end{enumerate}

Let $\mathbb C(q^{1/2}) \cdot v$ be a 1-dimensional vector space.
Let $\lambda \in P$, and set  $X_{\beta^+_r}\cdot v=0$, for all $r\in\mathbb Z$ and $E_{k\delta}^{(i)}\cdot v=0$
for $k<0$ and $1\leq i\leq N$ ,  $K_i^{\pm 1}\cdot v = q^{\pm\lambda(h_i)}v$
($i \in I$) and
$D^{\pm 1}\cdot v = q^{\pm \lambda(d)}v$.
Define $M_q^r(\lambda)=U_q(\g)\otimes_{B_q^r} \mathbb C(q^{1/2}) v$.  Then
$M_q^r(\lambda)$ is an $S$-highest weight $U_q$-module called
the {\it quantum imaginary Verma module} with highest weight $\lambda$.
If we let $L_q^r$ be the left ideal in $U_q$ generated by $X_{\beta_r^-} $ for all $r\in\mathbb Z$ and $E_{k\delta}^{(i)} $
for $k<0$ and $1\leq i\leq N$,  $K_i^{\pm 1}- q^{\pm\lambda(h_i)}$
($i \in I$) and
$D^{\pm 1}-q^{\pm \lambda(d)}$, then $U_q/L_q^r\cong M_q^r(\lambda)$ which is induced by $1\mapsto v$.

We obtain the following refinement of \cite[Theorem 3.5]{FGM05}:
\color{black}

\begin{thm}[\cite{CDFM13}]  \label{pbwthm}As a vector space, $M_q^r(\lambda)$ has a basis consisting of the ordered monomials
\begin{equation}\label{pbw}
\{X^-(\um)F^{{\text{im}}}(\ul)v\}.
\end{equation}
In particular,  $M_q^r(\lambda)$ is free as a module over $U_q(-S)$.
\end{thm}

%\begin{cor}[\cite{CFK2}]   $M_q^r(\lambda)$ is free as a module over $U_q(-S)$.
%\end{cor}
Recall the notation from \secref{drinfeldgen}.
Let $M^d_q(\lambda)=U_q/L_q^d$ where $L_q^d$ is the left ideal generated by the Drinfeld generators $x^+_{ik}$, $h_{il}$, $i\in  I_0$, $k\in\mathbb Z$, $l>0$, together with $K_i^{\pm 1}-q^{\pm \lambda(h_i)}$, $\gamma^{\pm 1/2}-q^{\pm \lambda(c)/2}$ and $D^{\pm 1}-q^{\pm \lambda(d)}$.    Let $B_q^d$ be the subalgebra of $U_q$ generated by $U^+_q(S)$ and $U^0_q(\hat{\mathfrak g})$ and let $\mathbb C(q^{1/2})_\lambda$ be the one dimensional $B_q^d$-module where $x^+_{ik}1=0$, $h_{il}1=0$, $K_i^{\pm 1}1=q^{\pm \lambda(h_i)}1$, $i\in  I_0$, $k\in\mathbb Z$, $l>0$,  $\gamma^{\pm 1/2}1=q^{\pm \lambda(c)/2}1$ and $D^{\pm 1}1=q^{\pm \lambda(d)}1$.  Note that $B_q^d\subseteq B_q^r$ as $E_{\alpha_i+k\delta}=o(i)^kx^+_{ik}$ for $k\geq0$,  $F_{-\alpha_i-k\delta}=-o(i)^kK_i \gamma^kx^+_{ik}$ for $k<0$, and $E_{k\delta}^{(i)}=\gamma^{-k/2}h_{ik}$ (see \eqnref{rootvector1} and  \eqnref{rootvector4}).

By universal mapping properties of quotients and the tensor products one has
$$
M_q^d(\lambda)\cong U_q\otimes _{B^d_q}\mathbb C(q^{1/2})_\lambda.
$$
Since $L^d_q\subset L^r_q$, there is a surjective $U_q$-module homomorphism $\pi: M^d_q(\lambda)\to M_q^r(\lambda)$. 
\begin{cor}  [\cite{CDFM13}] \label{cor-isom-two-imag}
$M^d_q(\lambda)$ is isomorphic to $M^r_q(\lambda)$ as $U_q$-modules. 
\end{cor}

We have  immediately from \cite{FGM05}, Corollary 6.5. 
\begin{cor}  
$M^d_q(\lambda)$ is irreducible if and only if $\lambda(c)\neq 0$. 
\end{cor}

\subsection{Reduced imaginary Verma modules}
Let $\lambda\in P$.  
Suppose now that $\lambda(c)=0$. Then $\gamma^{\pm \frac{1}{2}}$ acts on $M_q(\lambda)$ by $1$. 
Denote by $J^q(\lambda)$ the left ideal of
$U_q=U_q(\hat{\mathfrak g})$ generated by  $L_q^d$  and  $h_{il}$ for
all $l$ and all $i\in \dot I$.

 Set
$$
\tilde{M}_q(\lambda)=U_q/J^q(\lambda).
$$
Then $\tilde{M}_q(\lambda)$ is a homomorphic image of
 $M^d_q(\lambda)$ which we call the \emph{reduced imaginary Verma
 module}. The module $\tilde{M}_q(\lambda)$ has a $P$-gradation:
 $$
 \tilde{M}_q(\lambda)=\sum_{\xi\in P}\tilde{M}_q(\lambda)_{\xi},
  $$
   where $\tilde{M} _q(\lambda)_\xi$ is spanned by
 $$
 E_{-\beta_1-m_1\delta}\cdots E_{-\beta_l-m_l\delta}E_{-\gamma_1+k_1\delta}\cdots E_{-\gamma_r+k_r\delta}
        \quad m_i\geq 0,k_i>0\enspace \beta_i,\gamma_i\in \dot\Delta_+ \cdot 1$$
for
$$
\xi=-\sum_{i=1}^l\beta_i-\sum_{j=1}^r \gamma_j+\left(-\sum_{i=1}^lm_i+\sum_{j=1}^r k_j\right)\delta.
$$

Applying \cite[Theorem 7.1]{FGM05} and Corollary~\ref{cor-isom-two-imag} we obtain 

\begin{thm}\label{thm-imag-irred}
Let $\lambda\in P$ such that $\lambda(c)=0$. Then module $\tilde{M}_q(\lambda)$
is simple if and only if $\lambda(h_i)\neq 0$ for all $i\in \dot I$.
\end{thm}

\section{$\Omega$-operators and their relations}
 Recall that the Schur polynomials $S_k(\mathbf x)$, $k\in\mathbb Z$ are defined to be polynomials in $\mathbb C[x_1,x_2,\dots]$ given by 
$$
\sum_{k\in\mathbb Z}^\infty S_k(x)z^k=\exp\left(\sum_{l=1}^\infty x_lz^l\right)
$$

Consider now the subalgebra $ \mathcal N_q^-$, generated by $\gamma^{\pm1/2}$, and $x_{i,l}^-$, $l\in\mathbb Z$, $1\leq i\leq N$. Note that the corresponding relations \eqnref{Serre} hold in $ \mathcal N_q^-$.

\begin{lem} \label{Plemma} Fix $k\in\mathbb Z$ and $1\leq i\leq N$.  Then for any $P\in \mathcal N_q^-$, there exists unique
$$
Q(i,k,p),R(i,k,r)\in \mathcal N_q^-,\quad p,r\in\mathbb Z,
$$
  such that
\begin{equation}
[x^+_{i,k},P]=K_i\sum \frac{   S_{i,p}^+  Q(i,k,p)}{q_i-q^{-1}_i}+K_i^{-1}\sum \frac{S_{i,r}^- R(i,k,r)}{q_i-q_i^{-1}}.
\end{equation}
where 
\begin{align*}
S_{i,k}^+&:=S_k((q_i-q_i^{-1})E_{\delta}^{(i)},(q_i-q_i^{-1})E_{2\delta}^{(i)},\dots ) ,\\
   S_{i,k}^-&:=S_k((q_i-q_i^{-1})E_{-\delta}^{(i)},(q_i-q_i^{-1})E_{-2\delta}^{(i)},\dots ).
\end{align*}
Note that the $S_{i,k}$ have degree $k$ with respect to $D$. 
\end{lem}
\begin{proof}  For the existence we have the following:
Now any element in $\mathcal N_q^-$ is a sum  of elements of the form
$$
P_{m_1,\dots, m_k}=\gamma^{l/2}x^- _{j_1,m_1} \cdots x^- _{j_k,m_k} , 
$$
where 
$m_i\in\mathbb Z, k\geq 0,$ $ l\in\mathbb Z, \,1\leq j_i\leq N$
and such a product is a summand of
$$
P=P(v_1,\dots,v_k):=\gamma^{l/2}x_{j_1}^-(v_1)\cdots x_{j_k}^-(v_k)
$$
Set $\bar P=x_{j_1}^-(v_1)\cdots x_{j_k}^-(v_k)$ and $\bar P_l=x_{j_1}^-(v_{1})\cdots x_{j_{l-1}}^-(v_{l-1})x_{j_{l+1}}^-(v_{l+1})\cdots x_{j_k}^-(v_k)$.

Then we have by \eqnref{phix} and \eqnref{psix},
\begin{align*}
x_{j_1}^-(v_1)\cdots x_{j_{l-1}}^-(v_{l-1})\psi_i(v_l\gamma^{1/2})
	&=\prod_{m=1}^{l-1}g_{i,j_m}(v_{j_m}v_l^{-1} )^{-1}
	\psi_i(v_l\gamma^{1/2}) x_{j_1}^-(v_1)\cdots x_{j_{l-1}}^-(v_{l-1})\\
x_{j_1}^-(v_1)\cdots x_{j_{l-1}}^-(v_{l-1})\phi_i(u\gamma^{1/2})
	&=\prod_{m=1}^{l-1}
	g_{i,j_m}(u\gamma v_{j_m}^{-1})\phi_i(u\gamma^{1/2})x_{j_1}^-(v_1)\cdots x_{j_{l-1}}^-(v_{l-1}), \\
\end{align*}
so that by \eqnref{xx}
\begin{align*}
[x_i^+(u),&x_{j_1}^-(v_1)\cdots x_{j_k}^-(v_k)]=\sum_{l=1}^kx_{j_1}^-(v_1)\cdots [x_i^+(u),x_{j_l}^-(v_l)]\cdots x_{j_k}^-(v_k)  \\
&=\sum_{l=1}^k \delta_{i,j_l}x_{j_1}^-(v_1)\cdots\left (\frac{\delta(u/v_l\gamma)\psi_i(v_l\gamma^{1/2})-\delta(u\gamma/v_l)\phi_i(u\gamma^{1/2})}{q_i-q_i^{-1}}\right)\cdots x_{j_k}^-(v_k)  \\
&=\sum_{l=1}^k \delta_{i,j_l}x_{j_1}^-(v_1)\cdots x_{j_{l-1}}^-(v_{l-1})\psi_i(v_l\gamma^{1/2})x_{j_{l+1}}^-(v_{l+1})\cdots x_{j_k}^-(v_k) \frac{\delta(u/v_l\gamma)}{q_i-q_i^{-1}} \\
&\quad-\sum_{l=1}^k \delta_{i,j_l}x_{j_1}^-(v_1)\cdots x_{j_{l-1}}^-(v_{l-1})\phi_i(u\gamma^{1/2})x_{j_{l+1}}^-(v_{l+1})\cdots x_{j_k}^-(v_k) \frac{\delta(u\gamma/v_l)}{q_i-q_i^{-1}}\\  \\
&=\sum_{l=1}^k\delta_{i,j_l}\prod_{m=1}^{l-1}g_{i,j_m}(v_{j_m}v_l^{-1} )^{-1} \frac{\psi_i(v_l\gamma^{1/2})\delta(u/v_l\gamma)}{q-q^{-1}}  \bar P_l
\\
&\quad-\sum_{l=1}^k\delta_{i,j_l}\prod_{m=1}^{l-1} 
g_{i,j_m}(u\gamma v_{j_m}^{-1})\frac{\phi_i(u\gamma^{1/2})\delta(u\gamma/v_l)}{q_i-q_i^{-1}}  \bar P_l \\  \\
%%%%%%%%%%%%%%%%%%
&=\frac{\psi_i(u\gamma^{-1/2})}{q_i-q_i^{-1}} \sum_{l=1}^k\delta_{i,j_l}\prod_{j=1}^{l-1}g_{i,j_m,q^{-1}}(v_{j_m}/v_l)
  \bar P_l \delta(u/v_l\gamma)    \\
&\quad-\frac{\phi_i(u\gamma^{1/2})}{q_i-q_i^{-1}}
\sum_{l=1}^k\delta_{i,j_l}\prod_{j=1}^{l-1}
g_{i,j_m}(v_l/v_{j_m})\bar P_l\delta(u\gamma/v_l) .
\end{align*}

Note that $\psi_{i,k}(u\gamma^{-k/2})$ and $\phi_{i,k}(u\gamma^{k/2})$ do not depend on $P$.
By \eqnref{phidef} we can rewrite 
\begin{align*}
\psi_i(u\gamma^{-1/2})=\sum_{k=0}^{\infty}\psi_{ik}\gamma^{k/2}u^{-k} &= K_i \exp\left(
(q_i-q^{-1}_i)\sum_{l>0}  h_{il}\gamma^{l/2}u^{-l}\right)  \\
&=\quad K_i\left(\sum_{k=0}^\infty S_{i,k}^+u^{-k}\right),
\end{align*}
so that  $\psi_{i,k}\gamma^{k/2}=K_iS_{i,k}^+$ and similarly $\phi_{i,k}\gamma^{-k/2}=K_i^{-1}S_{i,k}^-$. Thus  
\begin{align*}
 [x^+_{im},x^-_{j_1,n_1}\cdots x^-_{j_k,n_k}]&=K_i\sum \frac{ S_{i,p}^+Q(i,k,p)}{q_i-q^{-1}_i}+K_i^{-1}\sum \frac{ S_{i,r}^-R(i,k,r)}{q_i-q_i^{-1}},
\end{align*}
where $Q(i,k,p),R(i,k,r)\in \mathcal N_q^-$.   This proves existence.

Uniqueness is proven as follows: 
 The components of $\bar P_l(u)$ have the form 
$$
x^-_{j_1,n_1}\cdots x^-_{j_{l-1},n_{l-1}}x^-_{j_{l+1},n_{l+1}}\cdots x^-_{j_k,n_k}
$$
and after using the Levendorski and So{\u\i}belman's convexity formula \eqnref{convexity} (possibly after applying a Lusztig automorphism $T_{w}$) we can rewrite this component as a linear combination of elements of the form $X^-(\mathbf m)$.  Let $\mathcal F^-$ be the span of the $X^-(\mathbf m)$ over $\mathbb Q(q^{1/2})$.    Then
\begin{align*}
 [x^+_{im},x^-_{j_1,n_1}\cdots x^-_{j_k,n_k}]&=K_i\sum \frac{ S_{i,p}^+\tilde Q(i,k,p)}{q_i-q^{-1}_i}+K_i^{-1}\sum \frac{ S_{i,r}^-\tilde R(i,k,r)}{q_i-q_i^{-1}}.
\end{align*}
where $\tilde Q(i,k,p),\tilde R(i,k,r)\in \mathcal F^-$. 

If also
\begin{align*}
 [x^+_{im},x^-_{j_1,n_1}\cdots x^-_{j_k,n_k}]&=K_i\sum \frac{ S_{i,p}^+ Q^\sharp(i,k,p)}{q_i-q^{-1}_i}+K_i^{-1}\sum \frac{ S_{i,r}^- R^\sharp(i,k,r)}{q_i-q_i^{-1}}.
\end{align*}
for some $ Q^\sharp(i,k,p), R^\sharp(i,k,r)\in \mathcal N_q^-$, then from \thmref{pbwthm} we must have 
$$
 Q^\sharp(i,k,p)=\tilde Q(i,k,p)= Q(i,k,p),\quad  R^\sharp(i,k,r)=\tilde R(i,k,r)= R(i,k,r)
$$
as the $S_{i,p}^+,S_{i,r}^-$ have leading terms $(E_{\pm\delta}^{(i)})^k$ that are distinct PBW basis elements and $\tilde Q(i,k,p)$ and the $\tilde R(i,k,r)$ are sums of PBW basis elements.
\end{proof}

\lemref{Plemma} motivates the definition of a family of operators as follows.
Set
\begin{align*}
G_{il}&=G_{il}^{1/q}=G_{il}^{1/q}(v_{j_1},\dots,v_{j_l},v_l):=\delta_{i,j_l}\prod_{j=1}^{l-1}g_{i,j_m,q^{-1}}(v_{j_m}/v_l), \\
 G_{il}^{q}&=G_{il} (v_{j_1},\dots,v_{j_l},v_l):=\delta_{i,j_l}\prod_{j=1}^{l-1}g_{i,j_m}(v_l /v_{j_m})
\end{align*}
where $G_{i1}:=\delta_{i,j_1}$.
Now define a collection of operators $\Omega_{\psi_i}(k),\Omega_{\phi_i}(k):\mathcal N_q^-\to \mathcal N_q^-$, $k\in\mathbb Z$, in terms of the generating functions
$$
\Omega_{\psi_i}(u)=\sum_{l\in\mathbb Z}\Omega_{\psi_i}(l)u^{-l},\quad \Omega_{\phi_i}(u)=\sum_{l\in\mathbb Z}\Omega_{\phi_i}(l)u^{-l}
$$
by
\begin{align}\label{definingomegapsi}
\Omega_{\psi_i}(u)(\bar P):&= \sum_{l=1}^kG_{il}
  \bar P_l \delta(u/v_l\gamma) \\
   \Omega_{\phi_i}(u)(\bar P):&=\sum_{l=1}^kG_{il}^q  \bar P_l\delta(u\gamma/v_l).\label{definingomegaphi}
\end{align}
%(\color{red} I don't recall where the power of $\gamma$ came from or if it is needed.  It appeared in our paper on $U_q(\hat{\mathfrak{sl}}(2,\mathbb C)$. \color{black})
Then we can write the above computation in the proof of Lemma 4.0.2 as
\begin{equation}\label{xplusP}
[x^+_i(u),\bar P]=(q_i-q_i^{-1})^{-1}\left({\psi_i}(u\gamma^{-1/2})\Omega_{\psi_i}(u)(\bar P)-{\phi_i}(u\gamma^{1/2})\Omega_{\phi_i}(u)(\bar P)\right).
\end{equation}

Note that $\Omega_{{\psi_i}}(u)(1)=\Omega_{{\phi_i}}(u)(1)=0$.  More explicitly let us write
$$
\bar P=x^-_{j_1}(v_1)\cdots x_{j_k}^-(v_k)=\sum_{n\in\mathbb Z}\sum_{\genfrac{}{}{0pt}{}{n_1,n_2,\dots,n_k\in\mathbb Z}{n_1+\cdots +n_k=n}}x^-_{j_1,n_1}\cdots x^-_{j_k,n_k}v_1^{-n_1}\cdots v_k^{-n_k}
$$
Then
\begin{align*}
\psi_i&(u\gamma^{-1/2})\Omega_{\psi_i}(u)(\bar P)\\
&=\sum_{l\geq 0}\sum_{p\in\mathbb Z}\sum_{n_i\in\mathbb Z } \gamma^{l/2}\psi_{il}\Omega_{\psi_i}(p)(x^-_{j_1,n_1}\cdots x^-_{j_k,n_k})v_1^{-n_1}\cdots v_k^{-n_k}u^{-l-p}  \\
&=\sum_{n_i\in\mathbb Z } \sum_{m\in\mathbb Z}\sum_{l\geq 0}\gamma^{l/2}\psi_{il}\Omega_{\psi_i}(m-l)(x^-_{j_1,n_1}\cdots x^-_{j_k,n_k})v_1^{-n_1}\cdots v_k^{-n_k}u^{-m}
\end{align*}
while
\begin{equation*}
[x^+_i(u),\bar P]=\sum_{m\in\mathbb Z}\sum_{n_1,n_2,\dots,n_k\in\mathbb Z } [x^+_{im},x^-_{j_1,n_1}\cdots x^-_{j_k,n_k}]v_1^{-n_1}\cdots v_k^{-n_k}u^{-m}.
\end{equation*}
Thus for a fixed $m$ and $k$-tuple $((j_1,n_1),\dots,(j_k,n_k))$ the sum
$$
\sum_{l\geq 0}\gamma^{l/2}\psi_{il}\Omega_{\psi_i}(m-l)(x^-_{j_1,n_1}\cdots x^-_{j_k,n_k})
$$
must be finite.   Hence
\begin{equation}\label{omegalocalfin}
\Omega_{\psi_i}(m-l)(x^-_{j_1,n_1}\cdots x^-_{j_k,n_k})=0,
\end{equation}
 for $l$ sufficiently large.

\begin{prop} \label{commutatorprop} Consider $x^-_i(v)=\sum_mx^-_{im}v^{-m}$ as a formal power series of left multiplication operators $x^-_{im}:\mathcal N_q^-\to \mathcal N_q^-$.  Then
\begin{align}
\Omega_{\psi_m}(u)x^-_i(v)&=\delta_{i,m}\delta(v\gamma/u)+g_{i,m,q^{-1}}(v\gamma/u)x^-_i(v)\Omega_{\psi_m}(u),
\label{omegapsi}\\
  \Omega_{\phi_m}(u)x^-_i(v)&=\delta_{i,m}\delta(u\gamma/v)+g_{i,m}(u\gamma/v)x^-_i(v)\Omega_{\phi_m}(u)\label{omegaphi}  \\
(q^{(\alpha_j|\alpha_k)}u_1-u_2)\Omega_{\psi_j}(u_1)\Omega_{\psi_k}(u_2)&=(u_1-q^{(\alpha_j|\alpha_k)}u_2)\Omega_{\psi_k}(u_2)\Omega_{\psi_j}(u_1) \label{psipsi} \\
(q^{(\alpha_j|\alpha_k)}u_1-u_2)\Omega_{\phi_j}(u_1)\Omega_{\phi_k}(u_2)&=(u_1-q^{(\alpha_j|\alpha_k)}u_2)\Omega_{\phi_k}(u_2)\Omega_{\phi_j}(u_1)   \label{phiphi} \\
(q^{(\alpha_j|\alpha_k)}\gamma^2u_1-u_2)\Omega_{\phi_j}(u_1)\Omega_{\psi_k}(u_2)&=(\gamma^2u_1-q^{(\alpha_j|\alpha_k)}u_2)\Omega_{\psi_k}(u_2)\Omega_{\phi_j}(u_1)\label{omegaphipsi}
\end{align}
\end{prop}

\begin{proof} Setting $\bar P=x^-_{j_1}(v_1)\cdots x^-_{j_k}(v_k)$ we get
\begin{align*}
\Omega_{\psi_m}(u)x^-_i(v)&(\bar P)=\delta_{i,m}
  x^-_{j_1}(v_1)\cdots x^-_{j_k}(v_k)\delta(u/v\gamma) \\
  &\quad +x^-_i(v)\sum_{l=1}^k g_{i,m,q^{-1}}(v/v_l)G_{ml}
   \bar P_l \delta(u/v_l\gamma)  \\
  &=
 \delta_{i,m}\bar P \delta(u/v\gamma) +x_i^-(v)g_{i,m,q^{-1}}(v\gamma /u)\Omega_{\psi_m}(u)\bar P.
\end{align*}
Similarly
\begin{align*}
\Omega_{\phi_m}(u)x^-_i(v)(\bar P)&
    =\delta_{i,m}
 x^-_{j_1}(v_1)\cdots x^-_{j_k}(v_k) \delta(u\gamma/v ) \\
      &\quad +x^-_i(v)\sum_{l=1}^k g_{i,m} (v_l/v)G_{ml}^{q}
     \bar P_l \delta(u\gamma/v_l )  \\
  &=\delta_{i,m}\bar P \delta(v/u\gamma) +x^-_i(v)g_{i,m} (u\gamma/v)\Omega_{\phi_m}(u)\bar P.
\end{align*}
One can prove \eqnref{psipsi} and \eqnref{phiphi} directly from their definitions, \eqnref{definingomegapsi} and \eqnref{definingomegaphi}, but there is another way (due to Kashiwara) to prove this identity and it goes as follows:
\begin{align*}
\Omega_{\psi_j}(u_1)\Omega_{\psi_k}(u_2)&x_i^-(v) =\delta_{k,i}
  \Omega_{\psi_j}(u_1) \delta(v\gamma/u_2) +\Omega_{\psi_j}(u_1) x_i^-(v)g_{i,k,q^{-1}}(v\gamma /u_2)\Omega_{\psi_k}(u_2)
  \\
  &=\delta_{k,i} \Omega_{\psi_j}(u_1) \delta(v\gamma/u_2) +\delta_{j,i}g_{i,k,q^{-1}}(v\gamma /u_2)\Omega_{\psi_k} (u_2)\delta(v\gamma/u_1)
 \\
  &\quad +g_{k,i,q^{-1}}(v\gamma /u_2)g_{j,i,q^{-1}}(v \gamma/u_1)x^-_i(v)\Omega_{\psi_j}(u_1)\Omega_{\psi_k}(u_2)
\end{align*}
and on the other hand
\begin{align*}
\Omega_{\psi_k}(u_2)\Omega_{\psi_j}(u_1)x_i^-(v) 
%  \Omega_\psi(u_2) \delta(v\gamma/u_1) +\Omega_\psi(u_2) x^-(v)g_{q^{-1}}(v\gamma /u_1)\Omega_{\psi_j}(u_1)
%  \\
  &= \delta_{j,i}\Omega_{\psi_k}(u_2) \delta(v\gamma/u_1) +\delta_{k,i}g_{i,j, q^{-1}}(v\gamma /u_1)\Omega_{\psi_k} (u_1)\delta(v \gamma/u_2)
 \\
  &\quad +g_{j,i,q^{-1}}(v\gamma /u_1)g_{k,i,q^{-1}}(v \gamma/u_2)x_i^-(v)\Omega_{\psi_k}(u_2)\Omega_{\psi_j} (u_1)
\end{align*}

Thus setting $S=(u_1-q^{-(\alpha_j|\alpha_k)}u_2)\Omega_{\psi_j}(u_1)\Omega_{\psi_k}(u_2)-(q^{-(\alpha_j|\alpha_k)}u_1-u_2)\Omega_{\psi_k}(u_2)\Omega_{\psi_j} (u_1)$ we get
\begin{align*}
Sx_i^-(v)
&=(u_1-q^{-(\alpha_j|\alpha_k)}u_2)\delta_{i,k} \Omega_{\psi_j}(u_1) \delta(v\gamma /u_2)\\
&\quad +(u_1-q^{-(\alpha_j|\alpha_k)}u_2)\delta_{j,i}g_{k,i,q^{-1}}(v\gamma /u_2)\Omega_{\psi_k} (u_2)\delta(v \gamma/u_1) \\
&\quad +(u_1-q^{-(\alpha_j|\alpha_k)}u_2)g_{k,i,q^{-1}}(v\gamma /u_2)g_{j,i,q^{-1}}(v \gamma/u_1)x^-_i(v)\Omega_{\psi_j}(u_1)\Omega_{\psi_k} (u_2)  \\
&\quad-(q^{-(\alpha_j|\alpha_k)}u_1-u_2)\delta_{j,i}\Omega_{\psi_k}(u_2) \delta(v\gamma/u_1) \\
&\quad-(q^{-(\alpha_j|\alpha_k)}u_1-u_2)
		\delta_{k,i}g_{j,i,q^{-1}}(v\gamma /u_1)\Omega_{\psi_j}(u_1)\delta(v \gamma/u_2) \\
&\quad -(q^{-(\alpha_j|\alpha_k)}u_1-u_2)g_{j,i,q^{-1}}(v\gamma /u_1)g_{q^{-1}}(v \gamma/u_2)x^-_i(v)\Omega_{\psi_k}(u_2)\Omega_{\psi_j} (u_1) \\ \\
&=\left((u_1-q^{-(\alpha_j|\alpha_k)}u_2)\delta_{k,i}-(q^{-(\alpha_j|\alpha_k)}u_1-u_2)\delta_{k,i}g_{j,i,q^{-1}}(v\gamma /u_1)\right)\Omega_{\psi_j} (u_1)\delta(v \gamma/u_2)  \\
&\quad+\left((u_1-q^{-(\alpha_j|\alpha_k)}u_2)\delta_{j,i}g_{k,i,q^{-1}}(v\gamma /u_2) -(q^{-(\alpha_j|\alpha_k)}u_1-u_2)\delta_{j,i}\right)\Omega_{\psi_k}(u_2) \delta(v\gamma/u_1) \\
&\quad+g_{k,i,q^{-1}}(v\gamma /u_2)g_{j,i,q^{-1}}(v\gamma /u_1)x^-_i(v) \\
&\hskip 20pt \times \left((u_1-q^{-(\alpha_j|\alpha_k)}u_2)\Omega_{\psi_j}(u_1)\Omega_{\psi_k} (u_2) -(q^{-(\alpha_j|\alpha_k)}u_1-u_2) ) 
	\Omega_		{\psi_k}(u_2)\Omega_{\psi_j} (u_1)\right) \\ \\
&=g_{k,i,q^{-1}}(v\gamma /u_2)g_{j,i,q^{-1}}(v\gamma /u_1)x^-_i(v)S
\end{align*}
Hence
$$
Sx^-_{j_1}(v_1)\cdots x^-_{j_n}(v_n)=  \prod_{i=1}^ng_{k,j_i,q^{-1}}(v_{j_i} \gamma/u_1)g_{j,j_i,q^{-1}}(v_{j_i}\gamma /u_2)x^-_{j_1}(v_1)\cdots x^-_{j_n}(v_n) S,
$$
which implies, after applying this to $1$,  that  $S=0$.

Next we have
\begin{align*}
\Omega_{\phi_j}(u_1)\Omega_{\phi_k}(u_2)x^-_i(v)    
	&= \delta_{i,k}\Omega_{\phi_j}(u_1) \delta(v/u_2 \gamma) +\delta_{j,i}g_{k,i}(u_2 \gamma/v)\Omega_{\phi_k} (u_2)\delta(v/u_1 \gamma) \\
  	&\quad +g_{k,i}(u_2 \gamma/v)g_{j,i}(u_1 \gamma/v)x^-_i(v)\Omega_{\phi_j}(u_1)\Omega_{\phi_k} (u_2)
\end{align*}
and on the other hand
\begin{align*}
\Omega_{\phi_k}(u_2)\Omega_{\phi_j}(u_1)x^-_i(v)  &= \delta_{j,i}\Omega_{\phi_k}(u_2) \delta(v/u_1 \gamma) +\delta_{k,i}g_{j,i}(u_1 \gamma/v)\Omega_{\phi_j} (u_1)\delta(v/u_2 \gamma)\\
&\quad +g_{j,i}(u_1 \gamma/v)g_{k,i}(u_2 \gamma/v)x^-_i(v)\Omega_{\phi_k}(u_2)\Omega_{\phi_j} (u_1)
\end{align*}

So if we set   $S=(u_1-q^{-(\alpha_j|\alpha_k)}u_2)\Omega_{\phi_j}(u_1)\Omega_{\phi_k}(u_2)-(q^{-(\alpha_j|\alpha_k)}u_1-u_2)\Omega_{\phi_k}(u_2)\Omega_{\phi_j} (u_1)$ we get
\begin{align*}
Sx^-_i(v)  &= (u_1-q^{-(\alpha_j|\alpha_k)}u_2)\delta_{k,i}\Omega_{\phi_j}(u_1) \delta(v /u_2 \gamma)\\
&\quad +(u_1-q^{-(\alpha_j|\alpha_k)}u_2)\delta_{j,i}g_{k,i}(u_2 \gamma/v)\Omega_{\phi_k} (u_2)\delta(v/u_1 \gamma)  \\
&\quad +(u_1-q^{-(\alpha_j|\alpha_k)}u_2)g_{k,i}(u_2 \gamma/v)g_{j,i}(u_1 \gamma/v)x^-_i(v)\Omega_{\phi_j}(u_1)\Omega_{\phi_k} (u_2)  \\
&\quad -(q^{-(\alpha_j|\alpha_k)}u_1-u_2)\delta_{i,j}\Omega_{\phi_k}(u_2) \delta(v/u_1 \gamma)\\ 
&\quad -(q^{-(\alpha_j|\alpha_k)}u_1-u_2)\delta_{i,k}g_{j,i}(u_1 \gamma/v)\Omega_{\phi_j} (u_1)\delta(v/u_2 \gamma)\\
&\quad -(q^{-(\alpha_j|\alpha_k)}u_1-u_2)g_{j,i}(u_1 \gamma/v)g_{k,i}(u_2 \gamma/v)x^-_i(v)\Omega_{\phi_k}(u_2)\Omega_{\phi_j} (u_1)
 \\ \\
 %%%%%%%%%%%%%%%%%%
&= \left((u_1-q^{-(\alpha_j|\alpha_k)}u_2) -(q^{-(\alpha_j|\alpha_k)}u_1-u_2)g_{j,i}(u_1 \gamma/v)\right)
\delta_{k,i}\Omega_{\phi_j} (u_1)\delta(v/u_2 \gamma)      \\
&\quad+\left((u_1-q^{-(\alpha_j|\alpha_k)}u_2)g_{k,i}(u_2 \gamma/v) -(q^{-(\alpha_j|\alpha_k)}u_1-u_2)\right)
\delta_{j,i} \Omega_{\phi_k} (u_2)\delta(v/u_1 \gamma)  \\
&\quad +g_{k,i}(u_2 \gamma/v)g_{j,i}(u_1 \gamma/v)x^-_i(v) \\
&\hskip 20pt \times\left((u_1-q^{-(\alpha_j|\alpha_k)}u_2)\Omega_{\phi_j}(u_1)\Omega_{\phi_k} (u_2) -(q^{-2}u_1-u_2) \Omega_{\phi_k}(u_2)\Omega_{\phi_j}(u_1)\right)
 \\ \\
  %%%%%%%%%%
&=   g_{k,i}(u_2 \gamma/v)g_{j,i}(u_1 \gamma/v)x^-_i(v)S.
  %%%%%%%%%%%%
\end{align*}
As in the calculation for \eqnref{psipsi} we get $S=0$.

Moreover
\begin{align*}
\Omega_{\phi_j}(u_1)\Omega_{\psi_k}(u_2)&x^-_i(v) =\delta_{k,i}
  \Omega_{\phi_j}(u_1) \delta(v\gamma/u_2) +\Omega_{\phi_j}(u_1) x^-_i(v)g_{k,i,q^{-1}}(v\gamma /u_2)\Omega_{\psi_k}(u_2)
  \\
  &=\delta_{k,i} \Omega_{\phi_j}(u_1) \delta(v\gamma/u_2) +\delta_{j,i}g_{k,i,q^{-1}}(v\gamma /u_2)\Omega_{\psi_k} (u_2)\delta(u_1\gamma/v)
 \\
  &\quad +g_{k,i,q^{-1}}(v\gamma /u_2)g_{j,i}(u_1\gamma/v)x^-_i(v)\Omega_{\phi_j}(u_1)\Omega_{\psi_k} (u_2)
\end{align*}

and
\begin{align*}
\Omega_{\psi_k}(u_2)\Omega_{\phi_j}(u_1)&x^-_i(v) =\delta_{j,i}
  \Omega_{\psi_k}(u_2) \delta(u_1\gamma/v) +\Omega_{\psi_k}(u_2) x^-_i(v)g_{j,i}(u_1\gamma/v)\Omega_{\phi_j}(u_1)
  \\
  &= \delta_{j,i}\Omega_{\psi_k}(u_2) \delta(u_1\gamma/v)  +\delta_{k,i}g_{j,i}(u_1\gamma/v)\Omega_{\phi_j}(u_1)\delta(v\gamma/u_2)
\\
  &\quad+g_{k,i,q^{-1}}(v\gamma /u_2)g_{j,i}(u_1\gamma/v)x^-_i(v)\Omega_{\psi_k}(u_2)\Omega_\phi (u_1)
\end{align*}

Set $S=(q^{(\alpha_j|\alpha_k)}\gamma^2u_1-u_2)\Omega_{\phi_j}(u_1)\Omega_{\psi_k}(u_2)-(\gamma^2u_1-q^{(\alpha_j|\alpha_k)}u_2)\Omega_{\psi_k}(u_2)\Omega_{\phi_j}(u_1) $.
Then
\begin{align*}
Sx^-_i(v)
 &= (q^{(\alpha_j|\alpha_k)}\gamma^2u_1-u_2)\delta_{k,i}\Omega_{\phi_j}(u_1) \delta(v\gamma/u_2)\\
 &\quad   +(q^{(\alpha_j|\alpha_k)}\gamma^2u_1-u_2)\delta_{j,i}g_{k,i,q^{-1}}(v\gamma /u_2)\Omega_{\psi_k} (u_2)\delta(u_1\gamma/v)
 \\
  &\quad +(q^{(\alpha_j|\alpha_k)}\gamma^2u_1- u_2)g_{k,i,q^{-1}}(v\gamma /u_2)g_{j,i}(u_1\gamma/v)x^-_i(v)\Omega_{\phi_j}(u_1)\Omega_{\psi_k} (u_2)  \\
  &\quad-( \gamma^2u_1-q^{(\alpha_j|\alpha_k)}u_2) \delta_{j,i}\Omega_{\psi_k}(u_2) \delta(u_1\gamma/v)\\
 &\quad   -( \gamma^2u_1-q^{(\alpha_j|\alpha_k)}u_2)\delta_{k,i}g_{j,i}(u_1\gamma/v)\Omega_{\phi_j}(u_1)\delta(v\gamma/u_2)
\\
  &\quad-( \gamma^2u_1-q^{(\alpha_j|\alpha_k)}u_2)g_{k,i,q^{-1}}(v\gamma /u_2)g_{j,i}(u_1\gamma/v)x^-_i(v)\Omega_{\psi_k}(u_2)\Omega_\phi (u_1)\\  \\
&=\left((q^{(\alpha_j|\alpha_k)}\gamma^2u_1-u_2) -(\gamma^2u_1-q^{(\alpha_j|\alpha_k)}u_2)g_{j,i}(u_1\gamma/v)\right)\delta_{k,i}\Omega_{\phi_j}(u_1)\delta(v\gamma/u_2)
 \\
  &\quad +\left((q^{(\alpha_j|\alpha_k)}\gamma^2u_1-u_2)g_{k,i,q^{-1}}(v\gamma /u_2)
    -(\gamma^2u_1-q^{(\alpha_j|\alpha_k)}u_2)\right)\delta_{j,i} \Omega_{\psi_k}(u_2) \delta(u_1\gamma/v)
\\
  &\quad+g_{k,i,q^{-1}}(v\gamma /u_2)g_{j,i}(u_1\gamma/v)x^-_i(v)\\
  &\hskip 20pt \times\left((q^{(\alpha_j|\alpha_k)}\gamma^2u_1-u_2)\Omega_{\phi_j}(u_1)\Omega_{\psi_k} (u_2) -(\gamma^2u_1-q^{(\alpha_j|\alpha_k)}u_2)\Omega_{\psi_k}(u_2)\Omega_\phi (u_1)\right)\\  \\
  &=g_{k,i,q^{-1}}(v\gamma /u_2)g_{j,i}(u_1\gamma/v)x^-_i(v)S.
\end{align*}
As in the previous calculations we get that $S=0$ and thus the last statement of the proposition hold.

\end{proof}

The identities (\ref{omegapsi}) , (\ref{omegaphi}) in \propref{commutatorprop} can be rewritten as
\begin{align}
(q^{(\alpha_i|\alpha_j)}v\gamma- u)\Omega_{\psi_j}(u)x^-_i(v)&=(q^{(\alpha_i|\alpha_j)}v\gamma- u)\delta_{i,j}\delta(v\gamma/u)+(q^{(\alpha_i|\alpha_j)}v\gamma -u)x^-_i(v)\Omega_{\psi_i}(u),
\label{omegapsi2}\\
(q^{(\alpha_i|\alpha_j)}v- u\gamma)  \Omega_{\phi_j}(u)x^-_i(v)&=(q^{(\alpha_i|\alpha_j)}v- u\gamma)\delta_{i,j}\delta(v/u\gamma)+( v-q^{(\alpha_i|\alpha_j)}u\gamma)x^-_i(v)\Omega_{\phi_j}(u)\label{omegaphi3}
\end{align}
which may be written out in terms of components as
\begin{align}
&q^{(\alpha_i|\alpha_k)}\gamma\Omega_{\psi_j}(m)x^-_{i,n+1}- \Omega_{\psi_j}(m+1)x^-_{i,n} \\
&\quad =(q^{(\alpha_i|\alpha_j)}\gamma-1)\delta_{i,j}\delta_{m,-n-1}+ \gamma x^-_{i,n+1}\Omega_{\psi_j}(m)-q^{(\alpha_i|\alpha_j)}x^-_{i,n}\Omega_{\psi_j}(m+1),
\label{omegapsi4}\\
&  q^{(\alpha_i|\alpha_j)}\Omega_{\phi_j}(m)x^-_{i,n+1}-  \gamma\Omega_{\phi_j}(m+1)x^-_{i,n} \\
&\qquad =(q^{(\alpha_i|\alpha_j)}- \gamma)\delta_{i,j}\delta_{m,-n-1}+ x^-_{i,n+1}\Omega_{\psi_j}(m)-q^{(\alpha_i|\alpha_j)}\gamma x^-_{i,n}\Omega_{\psi_j}(m+1),
\label{omegaphi5}
\end{align}

We can also write \eqnref{omegapsi} in terms of components and as operators on $\mathcal N_q^-$
\begin{equation}\label{omegapsi6}
    \Omega_{\psi_j}(k)x^-_{i,m}=\delta_{i,j}\delta_{k,-m}\gamma^{k}+\sum_{r\geq 0}g_{i,j,q^{-1}}(r)x^-_{i,m+r}\Omega_{\psi_j}(k-r)\gamma^{r}.
\end{equation}
    The sum on the right hand side turns into a finite sum when applied to an element in $\mathcal N_q^-$, due to \eqnref{omegalocalfin}.

  We also have by \eqnref{omegaphipsi}
  \begin{equation}\label{omegaphipsi2}
    \Omega_{\psi_i}(k)\Omega_{\phi_j}(m)= \sum_{r\geq 0}g_{i,j}(r)\gamma^{2r}\Omega_{\phi_j}(r+m)\Omega_{\psi_i}(k-r),
\end{equation}
as operators on $\mathcal N_q^-$.

\section{The Kashiwara algebra $\mathcal K_q$}  The Kashiwara algebra $\mathcal K_q$ is defined to be the $\mathbb F(q^{1/2})$-algebra with generators $\Omega_{\psi_j}(m),x_i^-(n),\gamma^{\pm 1/2}$, $m,n\in\mathbb Z$, $1\leq i,j\leq N$, where $\gamma^{\pm 1/2}$ are central and the defining relations are
\begin{align}
&q^{(\alpha_i|\alpha_j)}\gamma\Omega_{\psi_j}(m)x^-_{i,n+1}-  \Omega_{\psi_j}(m+1)x^-_{i,n} \\
&=(q^{(\alpha_i|\alpha_j)}\gamma-1)\delta_{i,j}\delta_{m,-n-1}+ \gamma x^-_{i,n+1}\Omega_{\psi_j}(m) -q^{(\alpha_i|\alpha_j)}x^-_{i,n}\Omega_{\psi_j}(m+1) \notag \\
&q^{(\alpha_i|\alpha_j)} \Omega_{\psi_i}(k+1)\Omega_{\psi_j}(l) -
\Omega_{\psi_j}(l)\Omega_{\psi_i}(k+1)\\  
&=  \Omega_{\psi_i}(k)\Omega_{\psi_j}(l+1)
    - q^{(\alpha_i|\alpha_j)}\Omega_{\psi_j}(l+1)\Omega_{\psi_i}(k)\notag\label{omegapsi3}
\end{align}
(which comes from \eqnref{omegapsi}, \eqnref{psipsi} written out in terms of components), and
\begin{equation}
x^{-}_{i,k+1}x^{-}_{j,l} - q^{- (\alpha_i|\alpha_j)}x^{-}_{j,l}x^{-}_{i,k+1}  = q^{- (\alpha_i|\alpha_j)}x^{-}_{i,k}x^{-}_{j,l+1}
    - x^{-}_{j,l+1}x^{-}_{i,k}\label{xminusreln}
\end{equation}
%(I (Ben) believe that this family of relations is needed but i can't pinpoint where it is needed.  It didn't appear in the case of $U_q(\hat{\frak{sl}}(\mathbb C)$)
%\color{black}
 together with
\[
\gamma^{1/2}\gamma^{-1/2}=1=\gamma^{-1/2}\gamma^{1/2}.
\]

\begin{lem}  The $\mathbb F(q^{1/2})$-linear map $\bar{\alpha}:\mathcal K_q\to \mathcal K_q$ given by
$$
\bar{\alpha}(\gamma^{\pm 1/2})=\gamma^{\pm 1/2},\quad \bar{\alpha}(x^-_{i,m})=\Omega_{\psi_i} (-m),\quad \bar{\alpha}(\Omega_{\psi_i} (m))=x^-_{i,-m}
$$
for all $m\in\mathbb Z$ is an involutive anti-automorphism.
\end{lem}
\begin{proof}
We have
\begin{align*}
\bar{\alpha}&\left(x^-_{i,k+1}x^-_{j,l} - q^{- (\alpha_i|\alpha_j)}x^-_{j,l}x^-_{i,k+1}  \right) \\
&=\Omega_{\psi_j}(-l)\Omega_{\psi_i}(-k-1) -q^{-(\alpha_i|\alpha_j)}\Omega_{\psi_i}(-k-1)\Omega_{\psi_j}(-l)  \\
&=q^{-(\alpha_i|\alpha_j)}\Omega_{\psi_j}(-l-1)\Omega_{\psi_i}(-k) -\Omega_{\psi_i}(-k)\Omega_{\psi_j}(-l-1)  \\
&=\bar{\alpha}\left(q^{-(\alpha_i|\alpha_j)}x^-_{i,k}x^-_{j,l+1}-x^-_{j,l+1}x^-_{i,k}\right)
\end{align*}
and
\begin{align*}
\bar{\alpha}&\left(q^{(\alpha_i|\alpha_j)}\gamma\Omega_{\psi_j}(m)x^-_{i,n+1}-\Omega_{\psi_j}(m+1)x^-_{i,n}\right) \\
    &=q^{(\alpha_i|\alpha_j)}\gamma\Omega_{\psi_i}(-n-1)x^-_{j,-m}-\Omega_{\psi_i}(-n)x^-_{j,-m-1}  \\
    &=(q^{(\alpha_i|\alpha_j)}\gamma-1)\delta_{i,j}\delta_{-m,n+1}+ \gamma x^-_{j,-m}\Omega_{\psi_i}(-n-1)-q^{(\alpha_i|\alpha_j)}x^-_{j,-m-1}\Omega_{\psi_i}(-n) \\
    &=\bar{\alpha}\left((q^{(\alpha_i|\alpha_j)}\gamma-1)\delta_{i,j}\delta_{m,-n-1}+ \gamma x^-_{i,n+1}\Omega_{\psi_j}(m)-q^{(\alpha_i|\alpha_j)}x^-_{i,n}\Omega_{\psi_j}(m+1)\right)
\end{align*}
\end{proof}

\begin{lem}  
 $\mathcal N_q^-$ is a left $\mathcal K_q$-module and 
$$\mathcal N_q^-\cong \mathcal K_q/\sum_{i=1}^N\sum_{k\in\mathbb Z}\mathcal K_q\Omega_{\psi_i}(k).$$
\end{lem}
\begin{proof}
Proposition~\ref{commutatorprop} implies that  $\mathcal N_q^-$ is a left $\mathcal K_q$-module.
We have an induced left $\mathcal K_q$-module epimomorphism from $\mathcal K_q$ to $\mathcal N_q^-$ which sends $1$ to $1$.  Since the $\Omega_{\psi_i}(k)$ annihilates $1$ for all $k$ and $1\leq i\leq N$, we get an induced left  $\mathcal K_q^-$-module epimomorphism
\[
\begin{CD}
\mathcal K_q/\sum_{i=1}^N\sum_{k\in\mathbb Z}\mathcal K_q\Omega_{\psi_i}(k)@>\eta>> \mathcal N_q^-
\end{CD}
\]
Let  $C$ denote the subalgebra of $\mathcal K_q$ generated by $x^-_{i,m},\gamma^{\pm1/2}$.  Then we have a surjective homomorphism
\[
\begin{CD}
C @>\mu>>\mathcal K_q/\sum_{i=1}^N\sum_{k\in\mathbb Z}\mathcal K_q\Omega_{\psi_i}(k)
\end{CD}
\]
The composition $\eta\circ \mu$ is surjective and since $\mathcal N_q^-$ is defined by generators $x^-_{i,n}$, $1\leq i\leq N$, $\gamma^{\pm 1/2}$ and relations \eqnref{Serre}, we get an induced map $\nu:\mathcal N_q^-\to C$ splitting the surjective map $\eta\circ \mu$. Since the composition $\nu\circ \eta\circ \mu$ is the identity, we get that $\eta\circ\mu$ is an isomorphism and thus $\eta$ is an isomorphism.
\end{proof}
\begin{prop}\label{form}
There is a unique symmetric form $(\enspace, \enspace)$ defined on $\mathcal N^-_q$ satisfying
$$
(x^-_{i,m}a,b)=(a,\Omega_{\psi_i}(-m)b),\quad (1,1)=1.
$$
\end{prop}
\begin{proof}  Using the anti-automorphism $\bar\alpha$ we can make $M=\text{Hom}(\mathcal N^-_q,\mathbb F(q^{1/2}))$ into a left $\mathcal K_q$-module by defining
\begin{gather*}
(x^-_{i,m}f)(a)=f(\Omega_{\psi_i}(-m)a),\enspace (\Omega_{\psi_i}(m)f)(a)=f(x^-_{i,-m}a), \\
(\gamma^{\pm 1/2}f)(a)=f(\gamma^{\pm 1/2}a).
\end{gather*}
for $a\in \mathcal N_q^-$ and $f\in M$.

Consider the element $\beta_0\in M$ satisfying $\beta_0(1)=1$ and
$$
\beta_0\left(\sum_{i=1}^N\sum_{k\in\mathbb Z}x^-_{i,m}\mathcal K_q\right)=0.
$$
Then $\Omega_{\psi_i}(m)\beta_0=0$ for any $m\in\mathbb Z$, $1\leq i\leq N$, we get an induced homomorphism of $\mathcal K_q$-modules 
$$
\bar\beta:\mathcal N_q^-\cong \mathcal K_q/\sum_{i=1}^N\sum_{m\in\mathbb Z}\mathcal K_q\Omega_{\psi_i}(m)\to M,
$$
where $\bar\beta(1)=\beta_0$.
Define the bilinear form $(\enspace,\enspace):\mathcal N_q^-\times \mathcal N_q^-:\to \mathbb F(q^{1/2})$ by
$$
(a,b)=(\bar\beta(a))(b)
$$
This form satisfies $(1,1)=1$ and
\begin{gather*}
(x^-_{i,m}a,b)=(a,\Omega_{\psi_i}(-m)b),\quad (\Omega_{\psi_i}(m)a,b)=(a,x^-_{i,-m}b), \\ (\gamma^{\pm/2}a,b)=(a,\gamma^{\pm    /2}b).
\end{gather*}
Since $\mathcal N_q^-$ is generated by $x^-_{i,m}$ and $\gamma^{\pm 1/2}$ we get that the form is the unique form satisfying these three conditions.  The form is symmetric since the form defined by $(a,b)'=(b,a)$ also satisfies the above conditions.
\end{proof}

\section{Simplicity of $\mathcal N_q^-$ as a $\mathcal K_q$-module}

We will show that $\mathcal N_q^-$ is simple as a module over
$\mathcal K_q$.

\begin{lem}\label{le-primitive}
Let $P\in \mathcal N_q^-$. If $\Omega_{\psi_i}(s)P=0$ for all $s\in
\mathbb Z$  and all $1\leq i\leq N$, then $P$ is a constant multiple of $1$.

\end{lem}

\begin{proof}
We may assume without loss of generality that $P$ is a homogeneous
element, say $P\in (\mathcal N_q^-)_{\lambda-\xi}$. We assume
that $\xi\neq 0$. Then $\xi=\sum_{i=1}^N n_i\alpha_i+m\delta$, $n_i\geq 0$, $\sum_i n_i^2\neq 0$,  $m\in \mathbb
Z$. Set $|\xi|=n = \sum_i n_i$. We shall prove the lemma by  induction on
$|\xi|$.

Suppose $|\xi|=1$. Then $P=x^-_{i,m}$ for some $i$ and
\begin{align*}
    \Omega_{\psi_j}(s)(x^-_{i,m})&=\delta_{i,j}\delta_{s,-m}\gamma^{s}+\sum_{r'\geq 0}g_{i,j,q^{-1'}}(r)x^-_{i,m+r'}\Omega_{\psi_j}(s-r')\gamma^{r'}1 \\
    &=\delta_{i,j}\delta_{s,-m}\gamma^{s}
\end{align*}
%\color{black}
Hence $\Omega_{\psi_i}(-m)(P)\neq 0$ unless $P=0$.

Suppose $|\xi|>1$.  We have by hypothesis $\Omega_{\psi_i}(l)(P)=0$ for any $l\in\mathbb Z$ and all $1\leq i\leq N$. Then we use 
\eqnref{omegaphipsi2} so that for all $k$, $m\in\mathbb Z$, and  $1\leq i,j\leq N$ we get
%\color{blue}
  \begin{equation}
    \Omega_{\psi_i}(k)\Omega_{\phi_j}(m)(P)= \sum_{r\geq 0}g_{i,j}(r)\gamma^{2r}\Omega_{\phi_j}(r+m)\Omega_{\psi_i}(k-r)(P)=0.
\end{equation}
%\color{black}

Hence by the induction hypothesis $\Omega_{\phi_i}(m)(P)=0$ as $\Omega_{\phi_i}(m)(P) \in (\mathcal N_q^-)_{\lambda-\xi+1}$.
Then $[x^+_{i,m},P]=0$ by \eqnref{xplusP}.

Consider the imaginary Verma module
$M_q^r(\lambda)$ with $\lambda(c)=0$ and choose $\lambda$ such that $\lambda(h_i)\neq 0$ for some $h_i\in\mathfrak h$. Then $\tilde{M}_q(\lambda)$ is the unique irreducible quotient of $M_q^r(\lambda)$
 and $v=Pv_{\lambda}$ is a nonzero element of the module $\tilde{M}_q(\lambda)$.
% Since $\gamma \tilde{v}_{\lambda}=q^{\lambda(c)}\tilde{v}_{\lambda}=1\tilde{v}_{\lambda}$, then
%$\Omega_{\phi}(s)(P)\tilde{v}_{\lambda}=\Omega_{\psi}(s)(P)\tilde{v}_{\lambda}=0$.

\color{black}
Thus
%by \eqnref{xplusP},
$$x^+_{i,s}v=[x^+_{i,s}, P]\tilde{v}_{\lambda}+ Px^+_{i,s} \tilde{v}_{\lambda}=0$$
for all $s\in \mathbb Z$ and all $1\leq i\leq N$.

Consider $V= \mathcal N_q^-v\subset \tilde{M}_q(\lambda)$. Then $V$ is a nonzero proper submodule
of $\tilde{M}_q(\lambda)$ which is a contradiction by Theorem~\ref{thm-imag-irred}.
This completes the proof.
\end{proof}

 Lemma~\ref{le-primitive} implies immediately the following result.

\begin{thm}
The algebra $\mathcal N_q^-$ is simple as a $\mathcal K_q$-module.
\end{thm}

\begin{cor}  The form $(\enspace,\enspace)$ defined in \propref{form} is non-degenerate.
\end{cor}
\begin{proof}  By \propref{form} the radical of the form $(\enspace,\enspace)$ is a  $\mathcal K_q$-submodule of $\mathcal N_q^-$ and since $(1,1)=1$, the radical must be zero.
\end{proof}

We remark that in \cite{MR1115118}, Kashiwara introduced the algebra $\mathcal{B}_q$ and showed that $U_q^-(\mathfrak{g})$ is a simple  $\mathcal{B}_q$-module. This in turn played an important role in showing the existence of crystal base for $U_q^-(\mathfrak{g})$, hence the standard Verma module. We expect to show in a future publication that the Kashiwara algebra $\mathcal K_q$ will play a similar role in constructing a crystal-like base for the reduced imaginary Verma module.

\bibliographystyle{alpha}
\bibliography{math}
\def\cprime{$'$}
%\providecommand{\bysame}{\leavevmode\hbox to3em{\hrulefill}\thinspace}
%\providecommand{\MR}{\relax\ifhmode\unskip\space\fi MR }
%% \MRhref is called by the amsart/book/proc definition of \MR.
%\providecommand{\MRhref}[2]{%
%  \href{http://www.ams.org/mathscinet-getitem?mr=#1}{#2}
%}
%\providecommand{\href}[2]{#2}
%\bibliographystyle{amsalpha}

\end{document}